\documentclass[11pt]{amsart}

\usepackage[margin=1in]{geometry}
\usepackage{amsmath,amssymb,amsthm,mathtools,mathrsfs}
\usepackage{enumitem}
\usepackage{float}
\usepackage{bm}
\usepackage{tikz-cd}
\usepackage[hidelinks]{hyperref}
\usepackage{aliascnt}
\usepackage[nameinlink,capitalize]{cleveref}

\newtheorem{theorem}{Theorem}[section]
\newaliascnt{proposition}{theorem}
\newtheorem{proposition}[proposition]{Proposition}
\aliascntresetthe{proposition}
\newaliascnt{lemma}{theorem}
\newtheorem{lemma}[lemma]{Lemma}
\aliascntresetthe{lemma}
\newaliascnt{corollary}{theorem}
\newtheorem{corollary}[corollary]{Corollary}
\aliascntresetthe{corollary}
\newaliascnt{conjecture}{theorem}

\aliascntresetthe{conjecture}
\theoremstyle{definition}
\newaliascnt{definition}{theorem}
\newtheorem{definition}[definition]{Definition}
\aliascntresetthe{definition}
\newaliascnt{example}{theorem}
\newtheorem{example}[example]{Example}
\aliascntresetthe{example}
\theoremstyle{remark}
\newaliascnt{remark}{theorem}
\newtheorem{remark}[remark]{Remark}
\aliascntresetthe{remark}

\newcommand{\TT}{\mathcal T}

\newcommand{\CC}{\mathbb C}
\newcommand{\ZZ}{\mathbb Z}
\newcommand{\PP}{\mathbb P}
\newcommand{\EE}{\mathbb E}
\newcommand{\one}{\mathbf 1}
\newcommand{\Pf}{\operatorname{Pf}}
\newcommand{\csm}{\operatorname{csm}}
\newcommand{\ssm}{\operatorname{ssm}}
\newcommand{\Sym}{\operatorname{Sym}}
\newcommand{\ASym}{\operatorname{ASym}}
\newcommand{\Rstar}{R^{*}}
\newcommand{\core}{\operatorname{core}_2}
\newcommand{\BG}{\operatorname{BG}}

\newcommand{\st}{\tilde{s}}

\DeclareMathOperator{\Hom}{Hom}
\DeclareMathOperator{\GL}{GL}
\def\AI#1{{\color{red} AI}}

\title[Probability and degeneracy locus formulas]{Probability-theoretic interpretation of \\ degeneracy locus formulas}

\author{Rich\'ard Rim\'anyi}
\address{Department of Mathematics, UNC Chapel Hill}

\begin{document}

\begin{abstract}
We give explicit formulas for the stable Segre–Schwartz–MacPherson classes of degeneracy loci of symmetric and skew-symmetric maps. Together with the known formulas for ordinary linear maps, this completes the stable SSM theory for the three classical types. We interpret these formulas probabilistically, in terms of the endpoint random partition of a stochastic rational six-vertex model: the skew classes are probabilities of Maya-dimer events, and the symmetric classes are Chebyshev moments of the random 2-core.
\end{abstract}

\maketitle

\section{Introduction and main results}\label{sec:intro}

The fundamental cohomology class of a subvariety $\Sigma$ in a smooth ambient space encodes geometric and enumerative information on the subvariety. In practice one often computes such fundamental classes as specializations of so-called {\em universal} degeneracy loci formulas, such as the Giambelli-Thom-Porteous formula or the J\'ozefiak-Lascoux-Pragacz and Harris-Tu formulas. The building blocks of such universal formulas are the Schur functions $s_\lambda$. Schur functions themselves can be defined as the equivariant fundamental classes of matrix Schubert varieties.

Much more geometric information about the subvariety $\Sigma$ is encoded by an
inhomogeneous refinement of its fundamental class: its
Chern--Schwartz--MacPherson class $\csm(\Sigma)$, or the closely related
Segre--Schwartz--MacPherson (SSM) class $\ssm(\Sigma)$.  Such classes are often calculated by
specializations of universal SSM formulas.  Both analogy with the preceding
paragraph and experience suggest that the building blocks of universal SSM
formulas are the stable SSM classes of matrix Schubert cells.  They were
introduced and studied in \cite{FR} and are denoted by $\st_\lambda$.  They are
inhomogeneous deformations of Schur functions,
\[
  \st_\lambda=s_\lambda+\text{terms of higher degree},
\]
and, in the completed symmetric-function ring, satisfy the remarkable identity
\begin{equation}\label{eq:sumone-intro}
  \sum_\lambda \st_\lambda=1.
\end{equation}
In \cite[Remark~8.8]{FR}, some positivity properties are discussed and, consequently, identity \eqref{eq:sumone-intro} is interpreted formally as a probability
distribution on partitions.

The purpose of this paper is to develop that probability-theory viewpoint. The development has one common probabilistic input (after a positive finite-alphabet specialization the functions $\st_\lambda$ become an honest probability distribution on partitions, namely the endpoint distribution of a stochastic rational six-vertex model) and three geometric applications, one for each classical type. We now state them.

\subsection*{The six-vertex random partition}
Fix $m\ge1$ and $x_1,\ldots,x_m\ge0$, and put
\[
 p_i=\frac{x_i}{1+x_i}\in[0,1).
\]
Consider the corresponding row-inhomogeneous stochastic rational six-vertex model with one path
entering from the left in every row and no paths entering from below.  Throughout
the paper we identify $m$-element sets $I=\{i_1<\cdots<i_m\}$ of positive
integers with partitions $\lambda$ of length at most $m$ by
\begin{equation}\label{eq:I-lambda-intro}
 i_a=\lambda_{m+1-a}+a,
 \qquad 1\le a\le m,
\end{equation}
and we write $\Lambda$ for the random partition attached in this way to the
random set of top exit positions of the model.
In the whole paper $\PP$ stands for probability and $\EE$ for expected value.

\begin{theorem}[Six-vertex realization]\label{thm:intro-prob}
For every partition $\lambda$ with $\ell(\lambda)\le m$,
\[
 \PP(\Lambda=\lambda)
 =\st_\lambda(x_1,\ldots,x_m).
\]
\end{theorem}

The model itself is standard, see, e.g. \cite{borodin2017integrable, KZJ}; the point of \cref{thm:intro-prob} is the
identification of its endpoint law with the stable SSM functions occurring in
\eqref{eq:sumone-intro}.  This is the basic probability dictionary underlying all
three applications.

\subsection*{Ordinary matrices: corank as a boundary event}
Let $k\le n$, and let
\[
 \Sigma^r_{k,n}
 =\{\varphi\in\Hom(\CC^k,\CC^n):\dim\ker\varphi=r\}
\]
be the exact-corank-$r$ stratum.  Theorem~9.1 of \cite{FR}, after the same
positive specialization as above, has the following probability interpretation.
Take the six-vertex model with $m=k$, and write its exit positions as
$i_1<\cdots<i_k$.  Then
\[
 \left.\ssm(\Sigma^r_{k,n})\right|_{\beta=0,\,\alpha_i=-x_i}
 =\PP\bigl(\#\{1\le a\le k:i_a>n\}=r\bigr),
\]
where $\alpha$ and $\beta$ are Chern roots of the factors of the symmetry group $\GL_k\times \GL_n$.
The random variable on the right is the number of occupied horizontal edges on
the right boundary of the $k\times n$ truncation of the six-vertex model.
Consequently, the closed locus of matrices of corank at least $r$ gives the
corresponding tail probability.  We prove these statements in
Proposition~\ref{prop:ordinary-corank-prob}.

\subsection*{Skew-symmetric matrices: corank as an event}
Let
$
 B_\lambda=\{\lambda_i-i:i\ge1\}\subset\ZZ
$
be the Maya set of $\lambda$, and let $w_j(\lambda)=\one_{j\in B_\lambda}$.  There
are two {\em dimerizations} of $\ZZ$, obtained by partitioning it into adjacent pairs;
set
\[
 d_0(\lambda)=\sum_{k\in\ZZ}|w_{2k}-w_{2k+1}| \ \in 2\ZZ_{\ge0},
 \qquad \qquad
 d_1(\lambda)=\sum_{k\in\ZZ}|w_{2k-1}-w_{2k}|\ \in 2\ZZ_{\ge0}+1.
\]
Thus $d_\epsilon$ counts dimers occupied at exactly one endpoint.  

Let $\ssm(\Sigma^\wedge_{\infty,r})$ denote the stable equivariant SSM class of the
skew-symmetric corank-$r$ orbits, the stabilization being taken along dimensions
of the same parity as $r$, see Section~\ref{sec:skew} for details.

\begin{theorem}[Skew Maya-dimer formula]\label{thm:intro-skew}
If $r\equiv\epsilon\pmod2$, then
\begin{equation}\label{eq:skew-main-intro}
 \ssm(\Sigma^\wedge_{\infty,r})
 =\sum_{d_\epsilon(\lambda)=r}\st_\lambda.
\end{equation}
In particular, every stable $\st$-coefficient is $0$ or $1$.
\end{theorem}

The formula proves the skew positivity and transpose-invariance
assertions of \cite[Conj.~6.2]{PR}, and strengthens the former to an explicit
$0/1$ rule.  Its probability interpretation is the following.

\begin{corollary}[Skew corank probability]\label{cor:skew-prob}
Under every positive finite-alphabet specialization,
\[
 \ssm(\Sigma^\wedge_{\infty,r})(x_1,\ldots,x_m)
 =\PP\bigl(d_\epsilon(\Lambda)=r\bigr).
\]
\end{corollary}

\subsection*{Symmetric matrices: Chebyshev moments of the 2-core}
The $2$-core of a partition is the partition that remains after one iteratively
removes rim dominoes.  Every $2$-core is a staircase partition $\delta_k$.  We write
\begin{equation}\label{eq:kappa-def}
 \core(\lambda)=\delta_{\kappa(\lambda)},
 \qquad
 \delta_k=(k,k-1,\ldots,1).
\end{equation}
Let $V_k(t)$ be the Chebyshev polynomial of the third kind, normalized by
\begin{equation}\label{eq:cheb-rec-intro}
 V_0(t)=1,\qquad V_1(t)=2t-1,\qquad
 V_k(t)=2tV_{k-1}(t)-V_{k-2}(t).
\end{equation}
For example we have
$
 V_2=4t^2-2t-1,
 V_3=8t^3-4t^2-4t+1
$.
Let $\ssm(\Sigma^S_{\infty,r})$ denote the stable equivariant SSM class of the symmetric corank-$r$ orbits, see Section~\ref{sec:symmetric-statement} for details.

\begin{theorem}[$2$-core/Chebyshev formula]\label{thm:intro-sym}
For the stable symmetric corank orbits,
\begin{equation*} \sum_{r\ge0}t^r\ssm(\Sigma^S_{\infty,r})
 =\sum_\lambda V_{\kappa(\lambda)}(t)\st_\lambda .
\end{equation*}
Equivalently,
\[
 [\st_\lambda]\ssm(\Sigma^S_{\infty,r})
 =[t^r]V_{\kappa(\lambda)}(t).
\]
Hence the coefficient of $\st_\lambda$ depends only on the $2$-core of
$\lambda$.
\end{theorem}

This formula proves, in particular, the symmetric transpose-invariance and sign
assertions of \cite[Conj.~6.2]{PR} and explains the repeated coefficients observed
there.  Again, the probability interpretation is immediate.

\begin{corollary}[Symmetric Chebyshev observable]\label{cor:sym-prob}
Under every positive finite-alphabet specialization,
\begin{equation*}
 \sum_{r\ge0}t^r\ssm(\Sigma^S_{\infty,r})(x_1,\ldots,x_m)
 =\EE\!\left[V_{\kappa(\Lambda)}(t)\right].
\end{equation*}
\end{corollary}

The contrast between Corollaries~\ref{cor:skew-prob} and~\ref{cor:sym-prob} is
one of the organizing points of the paper:
\[
\begin{array}{ccl}
    \Lambda^2 & : & \text{corank is an event,}
    \\
    S^2 & : &  \text{corank is a Chebyshev observable of the random $2$-core}.
\end{array}
\]

The proofs of \cref{thm:intro-skew,thm:intro-sym} use the explicit CSM formulas
of \cite{PR}.  The main symmetric-function work uses the modified Robbins
formalism of \cite{FH}.  It expresses the finite specializations of the
$\st$-functions in beta-number coordinates.  We refine this formalism by
recording the Maya-dimer defect in the skew case and the
two-runner charge in the symmetric case.  Removing the smallest beta number
then gives deletion recurrences for the corresponding refined sums.  After a
change of variables, these recurrences are solved by Pfaffian interpolation:
their values at zero, one, and reciprocal pairs determine the candidates, while
explicit residue identities cancel the apparent poles.  This produces a
dimer-refined Littlewood identity in the skew case and a charge-refined
Littlewood identity in the symmetric case.  Finally, symmetrizing the two charge
states converts the charge statistic into the Chebyshev polynomials, and a
comparison with the Pfaffians of \cite{PR} identifies the resulting series with
the SSM classes of the matrix orbits.

\smallskip

The paper is organized as follows.  After recalling SSM characteristic classes in Section~\ref{sec:char_classes}, in \cref{sec:probability} we prove the
six-vertex realization and derive the ordinary-matrix boundary statistic.  In
\cref{sec:skew} we prove the skew Maya-dimer formula.
In \cref{sec:symmetric-statement} we develop the $2$-core/Chebyshev formalism and
state its consequences.  The charge-refined Littlewood identity is proved in
\cref{sec:charge}.  In \cref{sec:pr-symmetric} we connect that identity with the
\cite{PR} formula for symmetric matrices.  We return to the probability
interpretation in \cref{sec:prob-consequences}.

\section{SSM characteristic classes and the \texorpdfstring{$\st$}{s-tilde}-functions}
\label{sec:char_classes}

\subsection{The CSM and SSM classes}\label{s:ssm}
Let $\Sigma$ be a constructible subset of a smooth complex algebraic variety $M$.
The Chern--Schwartz--MacPherson class and the closely related
Segre--Schwartz--MacPherson class are characteristic classes that extend
the ordinary total Chern class of a smooth variety to singular and non-closed
subsets. Some of the recent interest in these classes is motivated by their
connections with geometric representation theory, in particular with Maulik-Okounkov stable envelopes \cite{MO}; see, for example, \cite{RV,FR,AMSS1}.
In this section we recall only the basic facts that will be used below. For broader introductions to characteristic classes of singular varieties, see the survey \cite{SY}, the
lectures \cite{SchLectures}, and Chapters~5--7 of \cite{CMTS}.

Responding to conjectures of Grothendieck and Deligne, MacPherson proved
\cite{MacPherson} the existence and uniqueness of a natural transformation
\[
 c_*:F(-)\longrightarrow H_*(-)
\]
from the functor of $\ZZ$-valued constructible functions to homology, both viewed
as covariant functors for proper morphisms of complex algebraic varieties, such
that
\[
 c_*(1_X)=c(TX)\cap[X]
\]
when $X$ is smooth. The pushforward of constructible functions along a proper map
$f:X\to Y$ is defined by taking Euler characteristics of the fibers.

For a constructible subset $\Sigma\subset M$ of a smooth variety, we define the
cohomological classes
\[
 \csm(\Sigma\subset M):=PD\bigl(c_*(1_\Sigma)\bigr),\qquad
 \ssm(\Sigma\subset M):=\frac{\csm(\Sigma\subset M)}{c(TM)},
\]
where $PD$ denotes the usual duality identification (with Borel--Moore homology in the noncompact case). These concepts are extended to the equivariant settings (the algebraic group $G$ acts on $M$; $f$, $\Sigma$ are $G$-invariant) in \cite{OhmotoCamb}.
We suppress equivariance from the notation, and write simply
$\csm(\Sigma)$ and $\ssm(\Sigma)$ when the ambient variety is clear. In most of
our applications the ambient space is a representation $V$ of $G$.

Two elementary properties will be central for us. Additivity of constructible
functions gives
\[
 \ssm(\Sigma_1\cup\Sigma_2)
 =\ssm(\Sigma_1)+\ssm(\Sigma_2)-\ssm(\Sigma_1\cap\Sigma_2),
\]
and the normalization of the MacPherson transformation gives
\[
 \ssm(V\subset V)=1.
\]
Thus, for a finite decomposition of $V$ into constructible pieces, the sum of
their SSM classes is $1$. This formal resemblance to the additivity and
normalization axioms of probability is the starting point of the present paper, cf. \cite[Remark~8.8]{FR}.

\begin{remark}
Characteristic classes, for example SSM classes, have also been studied from the local,
singularity-theoretic point of view, in relation to Chern--Mather classes, local Euler obstructions, and sectional Euler characteristics, see \cite{ZhangDet, ZhangOrbit} and references therein. 
\end{remark}

\subsection{\texorpdfstring{$\st$}{s-tilde} functions}
Fix a partition $\lambda$. For $k\ge\ell(\lambda)$ and $n\ge\lambda_1+k$,
$\lambda$ indexes a full-rank matrix Schubert cell in
$\Hom(\CC^k,\CC^n)$ for the action of $\GL_k\times B_n^-$. After forgetting
the $B_n^-$-equivariance and passing to the stable limit, its $\GL_k$-equivariant
SSM class is the symmetric function $\st_\lambda$; see \cite{FR} for more details on the construction and stabilization.

The function $\st_\lambda$ is a completed symmetric function with Schur expansion
\begin{equation}\label{eq:st-schur-expansion}
 \st_\lambda=\sum_{\mu\supseteq\lambda}a_{\lambda,\mu}s_\mu,
 \qquad a_{\lambda,\lambda}=1.
\end{equation}
For example,
\[
\st_{\varnothing}
=1-s_1 +(s_2+s_{11}) -(s_3+2s_{21}+s_{111})
 +(s_4+3s_{31}+s_{22}+3s_{211}+s_{1111})+\ldots,
\]
and the initial Schur coefficients of several other $\st$ functions are shown in
Figure~\ref{fig:stilde_vs_s}.
\begin{figure}
\[
\begin{array}{c ||r|r|rr|rrr|rrrrr|}
  & \varnothing & (1) & (2) & (11) & (3) & (21) & (111)
  & (4) & (31) & (22) & (211) & (1111) \\
  \noalign{\hrule height 1.2pt}
  \varnothing
  & 1 & -1 & 1 & 1 & -1 & -2 & -1 & 1 & 3 & 1 & 3 & 1 \\
  \hline
  (1)
  &   & 1 & -2 & -2 & 3 & 5 & 3 & -4 & -9 & -3 & -9 & -4 \\
  \hline
  (2)
  &   &   & 1 &   & -3 & -2 &   & 6 & 7 & 2 & 3 &   \\
  (11)
  &   &   &   & 1 &   & -2 & -3 &   & 3 & 2 & 7 & 6 \\
  \hline
  (3)
  &   &   &   &   & 1 &   &   & -4 & -2 &   &   &   \\
  (21)
  &   &   &   &   &   & 1 &   &   & -3 & -3 & -3 &   \\
  (111)
  &   &   &   &   &   &   & 1 &   &   &   & -2 & -4 \\
  \hline
  (4)
  &   &   &   &   &   &   &   & 1 &   &   &   &   \\
  (31)
  &   &   &   &   &   &   &   &   & 1 &   &   &   \\
  (22)
  &   &   &   &   &   &   &   &   &   & 1 &   &   \\
  (211)
  &   &   &   &   &   &   &   &   &   &   & 1 &   \\
  (1111)
  &   &   &   &   &   &   &   &   &   &   &   & 1\\
  \hline
\end{array}
\]
\caption{Schur coefficients of $\st_\lambda$ through total degree $4$.}
\label{fig:stilde_vs_s}
\end{figure}

We will occasionally use the explicit stable formula of \cite[Def.~8.2]{FR}, as follows. For an integer sequence $\alpha=(\alpha_1,\ldots,\alpha_N)$, set
\[
 s_\alpha=\det\bigl(c_{\alpha_i-i+j}\bigr)_{1\le i,j\le N},
 \qquad c_0=1,\quad c_j=0\ \text{for }j<0,
\]
and define the operator $\mathcal S_{z_1,\ldots,z_N}$ by
\[
 \mathcal S_{z_1,\ldots,z_N}
 \bigl(z_1^{\alpha_1}\cdots z_N^{\alpha_N}\bigr)=s_\alpha,
\]
extended linearly to formal power series. Rational factors are expanded at the
origin. Then for $\lambda=(\lambda_1,\ldots,\lambda_k)$ we have
\begin{equation*}
 \st_\lambda
 =\mathcal S_{z_1,z_2,\ldots}\left(
   \prod_{i=1}^k\left(\frac{z_i}{1+z_i}\right)^{\lambda_i}
   \prod_{j=1}^{\infty}\prod_{i=1}^{j}
   \frac{1+z_i-z_j}{1+z_i}
 \right).
\end{equation*}
The infinite expression is interpreted coefficientwise: the coefficient of every
fixed Schur function stabilizes after finitely many variables.

We record three structural properties that will be useful below. Let $\omega$
denote the standard involution of the ring of symmetric functions,
$\omega(s_\mu)=s_{\mu^T}$.

\begin{proposition}\label{prop:st-basic}
With the notation of \eqref{eq:st-schur-expansion}, the following hold.
\begin{enumerate}
 \item\label{item:st-sign} The Schur coefficients have alternating signs:
 \[
   (-1)^{|\mu|-|\lambda|}a_{\lambda,\mu}\ge0.
 \]
 \item\label{item:st-transpose} Transposition is compatible with the standard
 involution:
 \[
   \omega(\st_\lambda)=\st_{\lambda^T}.
 \]
 \item\label{item:st-row} For $r\ge0$,
 \[
 [s_{(r)}]\st_\lambda=
 \begin{cases}
   (-1)^{r-k}\binom{r}{k},&\lambda=(k),\quad 0\le k\le r,\\
   0,&\text{otherwise}.
 \end{cases}
 \]
\end{enumerate}
\end{proposition}

\begin{proof}
Part~\ref{item:st-sign} was conjectured in \cite[Conj.~8.4]{FR} and is proved in
\cite[\S4.1]{AMSS}.

For Part~\ref{item:st-transpose}, duality of Grassmannians exchanges the Schubert
cell indexed by $\lambda$ with the one indexed by $\lambda^T$. More concretely,
after choosing a nondegenerate pairing, the orthogonal-complement map
$U\mapsto U^\perp$ identifies the corresponding Grassmannians and, with the
opposite dual flag, carries the Schubert condition for $\lambda$ to that for
$\lambda^T$. On stable Schubert classes the induced operation is precisely
$\omega$.

For Part~\ref{item:st-row}, the support condition $\mu\supseteq\lambda$ in
\eqref{eq:st-schur-expansion} shows that a one-row Schur function can occur only
when $\lambda=(k)$ is itself a row. Specializing to one variable gives
\[
 \st_{(k)}(x)=\frac{x^k}{(1+x)^{k+1}}
 =\sum_{r\ge k}(-1)^{r-k}\binom{r}{k}x^r,
\]
and $s_{(r)}(x)=x^r$, which proves the formula.
\end{proof}

Most importantly for us, additivity of SSM classes survives stabilization.
For fixed $k$ and $n$, the full-rank matrix Schubert cells form a constructible
decomposition of the full-rank locus in $\Hom(\CC^k,\CC^n)$. The complement has
codimension $n-k+1$, so as $n\to\infty$ its SSM class disappears in every
fixed degree. Letting subsequently $k\to\infty$ therefore gives the
coefficientwise identity
\begin{equation}\label{eq:sum-st-section2}
 \sum_\lambda \st_\lambda=1.
\end{equation}
The properties in \cref{prop:st-basic} and the identity
\eqref{eq:sum-st-section2} can already be seen in the initial coefficients of
Figure~\ref{fig:stilde_vs_s}.

\subsection{Finite specializations and stable limits}\label{ss:stable-limits}
The classes appearing in \cref{thm:intro-skew,thm:intro-sym} are stable limits of
finite-rank classes, and we fix here the (elementary) framework in which those
limits are taken.

Let $\hat\Lambda$ be the completed ring of symmetric functions, in which the
$\st_\lambda$ live, and let $\hat\Lambda_m$ be the completed ring of symmetric
functions in $m$ variables $u_1,\ldots,u_m$.
For $m\le m'$ write
\[
 \rho_m:\hat\Lambda\to\hat\Lambda_m,
 \qquad
 \rho^{m'}_m:\hat\Lambda_{m'}\to\hat\Lambda_m
\]
for the specializations $u_{m+1}=u_{m+2}=\cdots=0$, respectively
$u_{m+1}=\cdots=u_{m'}=0$.  Under the identification $\hat\Lambda_m
=\hat H^*_{\GL_m}(\mathrm{pt})$, the map $\rho^{m'}_m$ is the restriction
induced by $\GL_m\hookrightarrow\GL_{m'}$, $g\mapsto\operatorname{diag}(g,I_{m'-m})$.

Since $\rho_m(s_\mu)=s_\mu(u_1,\ldots,u_m)$ vanishes exactly when $\ell(\mu)>m$,
and since the Schur expansion~\eqref{eq:st-schur-expansion} of $\st_\lambda$ is
supported on partitions $\mu\supseteq\lambda$, we get
\begin{equation}\label{eq:rho-st}
 \rho_m(\st_\lambda)=\st_\lambda(u_1,\ldots,u_m),
 \qquad
 \rho_m(\st_\lambda)=0\quad\text{when }\ell(\lambda)>m.
\end{equation}
Moreover $\st_\lambda=s_\lambda+(\text{higher degree})$, so in each fixed degree
$N$ only the finitely many $\lambda$ with $|\lambda|\le N$ contribute.  Hence for
arbitrary constants $c_\lambda$ the formal sum $\sum_\lambda c_\lambda\st_\lambda$
is a well-defined element of $\hat\Lambda$, and by \eqref{eq:rho-st}
\begin{equation}\label{eq:rho-of-sum}
 \rho_m\Bigl(\sum_\lambda c_\lambda\st_\lambda\Bigr)
 =\sum_{\ell(\lambda)\le m}c_\lambda\,\st_\lambda(u_1,\ldots,u_m).
\end{equation}
Finally, $\rho_m$ restricts to an isomorphism in each degree $N\le m$, and
therefore $\hat\Lambda=\varprojlim_m\hat\Lambda_m$ degreewise.  Consequently, if
$S\subseteq\ZZ_{\ge0}$ is infinite and $(f_m)_{m\in S}$ satisfies
$\rho^{m'}_m(f_{m'})=f_m$ for all $m\le m'$ in $S$, then there is a unique
$f\in\hat\Lambda$ with $\rho_m(f)=f_m$ for all $m\in S$.  We call $f$ the
{\em stable limit} of the family $(f_m)$.  This is the only sense in which
``stabilization'' is used below.

\section{Stable SSM functions and a stochastic six-vertex measure}\label{sec:probability}

\subsection{Finite evaluations of \texorpdfstring{$\st$}{s-tilde} functions}
We will frequently evaluate completed symmetric functions on a finite positive
alphabet $x_1,\ldots,x_m\ge0$ by the specialization
\begin{equation}\label{eq:finite-alphabet}
 1+c_1t+c_2t^2+\cdots
 =\prod_{i=1}^m\frac{1}{1-x_it}.
\end{equation}
For $m=1$ one has
\[
 \st_{(r)}(x)=\frac{x^r}{(1+x)^{r+1}}.
\]
Writing $p=x/(1+x)$, this becomes $(1-p)p^r$. Thus the one-row random
partition of \cref{thm:intro-prob} is simply a geometric random variable:
{\em the probability of seeing exactly $r$ heads before the first tail is $\st_{(r)}$, where for one toss $\PP(heads)=p$}.

The two-variable specialization is already less trivial. If
$\lambda=(a,b)$ with $a\ge b\ge0$, put
\[
 p_1=\frac{x_1}{1+x_1},\qquad p_2=\frac{x_2}{1+x_2},\qquad d=a-b,
\]
and let $h_d(p_1,p_2)=p_1^d+p_1^{d-1}p_2+\cdots+p_2^d$, with $h_{-1}=0$. Then
\begin{equation*}
 \st_{(a,b)}(x_1,x_2)
 =(1-p_1)(1-p_2)(p_1p_2)^b
 \bigl((1+p_1p_2)h_d(p_1,p_2)-2p_1p_2\,h_{d-1}(p_1,p_2)\bigr).
\end{equation*}
In the homogeneous specialization $x_1=x_2$, hence $p_1=p_2=:p$, this simplifies to
\[
 \st_{(a,b)}(x_1,x_1)
 =(1-p)^2p^{a+b}
 \bigl((a-b+1)(1-p)^2+2p\bigr).
\]
These formulas foreshadow the positivity that will become transparent from the
six-vertex realization in the next subsection.

More generally, specializing the weight-function formula and the SSM identification in \cite[Def.~4.1 and Cor.~5.2]{FR}, the finite evaluation of $\st$ is
\begin{equation}\label{eq:tilde-from-P}
 \st_\lambda(x_1,\ldots,x_m)
 =\frac{P_I(x_1,\ldots,x_m)}{\prod_{u=1}^m(1+x_u)^n},
\end{equation}
where
$I=\{i_1<\cdots<i_m\}\subset\{1,\ldots,n\}$
corresponds to $\lambda$ as in \eqref{eq:I-lambda-intro}, and
\begin{equation}\label{eq:P-I}
 P_I(x_1,\ldots,x_m)
 =\Sym\left[
 \prod_{u=1}^m x_u^{i_u-1}(1+x_u)^{n-i_u}
 \prod_{u<v}\frac{1+x_v-x_u}{x_v-x_u}
 \right].
\end{equation}
Here $\Sym$ denotes symmetrization over the assignments of the variables
$x_1,\ldots,x_m$ to the exit positions $i_1<\cdots<i_m$.

\begin{remark}\label{rem:n-independence}
The left-hand side of \eqref{eq:tilde-from-P} does not depend on $n$, while its
right-hand side appears to.  It does not: the factor $\prod_{u=1}^m(1+x_u)^n$ is
symmetric, hence may be pulled out of the symmetrization in \eqref{eq:P-I}, and
\[
 \frac{P_I(x_1,\ldots,x_m)}{\prod_{u=1}^m(1+x_u)^n}
 =\Sym\left[
 \prod_{u=1}^m\frac{x_u^{i_u-1}}{(1+x_u)^{i_u}}
 \prod_{u<v}\frac{1+x_v-x_u}{x_v-x_u}
 \right],
\]
which involves only $I$.  Thus \eqref{eq:tilde-from-P} holds for every
$n\ge i_m$.
\end{remark}

\subsection{A matrix-Schubert partition function}
Consider the rational six-vertex model with $R$-matrix
\begin{equation}\label{eq:R-unnormalized}
 R(x)=
 \begin{pmatrix}
 1+x&0&0&0\\
 0&x&1&0\\
 0&1&x&0\\
 0&0&0&1+x
 \end{pmatrix}
\end{equation}
in the basis $00,01,10,11$.
\begin{figure}
\[
\begin{tikzpicture}[scale=.6, thick]
  \foreach \x in {1,...,6}
    \draw[gray!40] (\x,1) -- (\x,3);
  \foreach \y in {1,2,3}
    \draw[gray!40] (1,\y) -- (6,\y);
  \foreach \x in {1,...,6}
    \foreach \y in {1,2,3}
      \fill[gray!70] (\x,\y) circle (1.2pt);
  \foreach \y in {1,2,3}
    \node[left] at (-0.3,\y) {row \y};
  \foreach \x in {1,...,6}
    \node[below] at (\x,0.7) {\x};
  \draw[->, red,very thick]  (0,1) -- (1,1) -- (1,3) -- (1,3.6);
  \draw[->, blue, very thick] (0,2) -- (3,2) -- (3,3) -- (3,3.6);
  \draw[->, teal,very thick] (0,3) -- (6,3) -- (6,3.6);
\end{tikzpicture}
\]
\caption{A sample path configuration with three incoming paths and top exit set $I=\{1,3,6\}$.}
\label{fig:6vertex}
\end{figure}
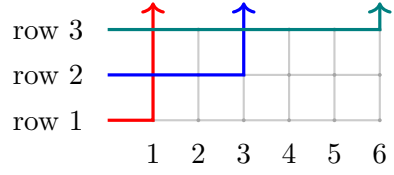

Namely, consider a grid with $m$ rows and $n\gg 0$ columns. Paths enter the grid on the left. At each vertex of the grid 0 or 1 or 2 paths enter from the left and/or from down, and paths exit to the right and/or up. To such a choice a weight is assigned, according to the entries of the $R$-matrix (for example 00 means no path enters from left or below, 01 means no path from left but path from below, etc). In row $u$ we use the variable $x_u$ for $x$ in the $R$-matrix. Instead of path/no path we use the terminology {\em occupancy/no occupancy} as well.

That is, at a vertex in row $u$, the two incoming bits record
the occupancies of the horizontal edge from the left and the vertical edge from
below, while the two outgoing bits record the occupancies to the right and above.
A nonzero matrix entry therefore preserves the total number of occupied edges;
these are precisely the six allowed local vertex configurations. When the two
incoming bits agree, the outgoing bits are forced and the weight is $1+x_u$.
When they differ, the occupancies may either remain in place, with weight $x_u$,
or interchange, with weight $1$.

In our model we put one occupied horizontal edge on the left
boundary of each row, no occupied vertical edges on the bottom boundary, and
require the occupied top edges to be exactly those in the columns $I$, see Figure~\ref{fig:6vertex}.
As always, $I=\{i_1<\cdots<i_m\}\subset\{1,\ldots,n\}$
corresponds to $\lambda$ as in~\eqref{eq:I-lambda-intro}.

Since $|I|=m$, conservation forces all right-boundary edges to be unoccupied. The
weight of a configuration is the product of its local vertex weights, with row
$u$ carrying parameter $x_u$, and $Z_I(x_1,\ldots,x_m)$, the partition function, denotes the sum of these weights over all admissible configurations.

\begin{proposition}\label{prop:bethe}
We have
\begin{equation}\label{eq:bethe}
 Z_I(x_1,\ldots,x_m)
 =\Sym\left[
 \prod_{u=1}^m x_u^{i_u-1}(1+x_u)^{n-i_u}
 \prod_{u<v}\frac{1+x_v-x_u}{x_v-x_u}
 \right].
\end{equation}
Consequently $Z_I=P_I$.
\end{proposition}

\begin{proof}
This is the coordinate Bethe-ansatz expression for the rational six-vertex model
with partial domain-wall boundary conditions; see, for example, \cite{FW,MP}.
For the ordered exits $i_1<\cdots<i_m$, the one-particle propagation factor in
row $u$ is $x_u^{i_u-1}(1+x_u)^{n-i_u}$. Exchanging two row parameters produces
the two-body scattering factor
\[
 \frac{1+x_v-x_u}{x_v-x_u}.
\]
Summing over the assignments of the row parameters to the ordered exits is the
symmetrization in \eqref{eq:bethe}. This gives the displayed formula, which is
identical to \eqref{eq:P-I}.
\end{proof}

Since every local weight in \eqref{eq:R-unnormalized} belongs to
$\{1,x_u,1+x_u\}$, the partition-function interpretation immediately gives the
following positivity statement.

\begin{corollary}\label{cor:monomial-positive}
For every $I$,
$P_I(x_1,\ldots,x_m)\in\ZZ_{\ge0}[x_1,\ldots,x_m]$.
\end{corollary}

\subsection{Stochastic normalization}
Put
$
 p_u={x_u}/(1+x_u)
$.
Dividing every row-$u$ local weight by $1+x_u$ turns
\eqref{eq:R-unnormalized} into
\begin{equation}\label{eq:R-stochastic}
 R_{p_u}
 =\begin{pmatrix}
 1&0&0&0\\
 0&p_u&1-p_u&0\\
 0&1-p_u&p_u&0\\
 0&0&0&1
 \end{pmatrix}
 =p_u\operatorname{Id}+(1-p_u)\TT,
\end{equation}
where $\operatorname{Id}$ is the identity matrix and $\TT$ is the transposition
matrix interchanging the two middle basis vectors $01$ and $10$.
Thus the local weights are stochastic: if the two incoming occupancies agree,
the transition is forced, while if they differ, the two bits remain in place
with probability $p_u$ and are interchanged with probability $1-p_u$.

For the probability model it is convenient to use a semi-infinite strip, with
columns indexed by $1,2,\ldots$, one occupied horizontal edge entering each row
from the left, and no occupied vertical edges entering from below. The finite
$m\times n$ rectangle above is then simply a truncation used to calculate the
probability of a prescribed finite set of top exits.

We are ready to prove the first theorem from the Introduction.

\begin{proof}[Proof of \cref{thm:intro-prob}]
Fix a top exit set $I=\{i_1<\cdots<i_m\}$ and choose $n\ge i_m$. On the event
that the top exits are exactly $I$, conservation implies that every horizontal
edge to the right of column $n$ is empty. Hence its probability in the
semi-infinite stochastic model is already computed in the $m\times n$ rectangle,
and equals
\[
 \frac{Z_I(x_1,\ldots,x_m)}{\prod_{u=1}^m(1+x_u)^n}.
\]
By \cref{prop:bethe} and \eqref{eq:tilde-from-P}, this is precisely
$\st_\lambda(x_1,\ldots,x_m)$.

It remains only to note that the model has exactly $m$ top exits almost surely.
Indeed, inductively in the rows, only finitely many vertical sites are occupied
before a new row is scanned. Once an occupied horizontal path has moved beyond
the rightmost of these sites, it encounters only empty vertical edges and turns
up at each successive column with probability $1-p_u>0$. Since $p_u<1$, it
turns up after finitely many further columns almost surely. Thus the possible
finite exit sets $I$ exhaust the sample space, and summing their probabilities
gives $1$.
\end{proof}

Consider the $\GL_k \times \GL_n$ representation $\Hom(\CC^k,\CC^n)$, and denote the Chern roots of the (standard representations of the) two factors of the group by $\alpha_1,\ldots,\alpha_k$ and $\beta_1,\ldots,\beta_n$.

\begin{proposition}[Ordinary corank as a boundary statistic]
\label{prop:ordinary-corank-prob}
Let $k\le n$, put $d=n-k$, and let
\[
 \Sigma^r_{k,n}
 =\{\varphi\in\Hom(\CC^k,\CC^n):\dim\ker\varphi=r\},
 \qquad 0\le r\le k.
\]
For a partition $\lambda$, define
\[
 h_d(\lambda)=\#\{j\ge1:\lambda_j\ge j+d\},
\]
the number of boxes on the $d$-shifted diagonal. Then
\begin{equation}\label{eq:FR-corank-diagonal}
 \ssm(\Sigma^r_{k,n})
 =\rho_{k,n}\!\left(
   \sum_{h_{n-k}(\lambda)=r}\st_\lambda
 \right),
\end{equation}
where $\rho_{k,n}$ is the supersymmetric specialization defined by
\[
 \rho_{k,n}(1+c_1t+c_2t^2+\cdots)
 =\frac{\prod_{j=1}^n(1+\beta_jt)}
        {\prod_{i=1}^k(1+\alpha_it)}.
\]
Under the positive finite-alphabet specialization
$\beta_j=0$ and $\alpha_i=-x_i$, where $x_1,\ldots,x_k\ge0$, one consequently has
\begin{equation}\label{eq:ordinary-corank-prob}
 \left.\ssm(\Sigma^r_{k,n})\right|_{\beta=0,\,\alpha_i=-x_i}
 =\PP\bigl(h_{n-k}(\Lambda)=r\bigr)
 =\PP\bigl(\#\{1\le a\le k:i_a>n\}=r\bigr).
\end{equation}
The last random variable is the number of occupied horizontal edges on the
right boundary of the $k\times n$ truncation of the six-vertex model.
\end{proposition}

\begin{proof}
By \cite[Theorem~9.1]{FR}, the sum in \eqref{eq:FR-corank-diagonal} is indexed by
the partitions satisfying
\[
 \lambda_r\ge r+d,
 \qquad
 \lambda_{r+1}\le r+d.
\]
Here the first condition is vacuous when $r=0$.  These conditions say precisely
that $h_d(\lambda)=r$.  After setting
$\beta_j=0$ and $\alpha_i=-x_i$, the defining series for $\rho_{k,n}$ becomes
\[
 \prod_{i=1}^k\frac{1}{1-x_it},
\]
which is the specialization \eqref{eq:finite-alphabet}.  The first equality in
\eqref{eq:ordinary-corank-prob} now follows from \cref{thm:intro-prob}.

For the second, recall that the exit positions and the endpoint partition are
related by $i_a=\lambda_{k+1-a}+a$.  Writing $j=k+1-a$, we have
\[
 \lambda_j\ge j+(n-k)
 \quad\Longleftrightarrow\quad
 i_a\ge n+1.
\]
Thus $h_{n-k}(\Lambda)=\#\{1\le a\le k:i_a>n\}$.  At the vertical cut after column $n$,
this is exactly the number of paths that have not yet exited through the top,
or equivalently the number of occupied horizontal right-boundary edges.
\end{proof}

By additivity, the closed determinantal locus
$
 D^{\ge r}_{k,n}
 =\{\varphi\in\Hom(\CC^k,\CC^n):\dim\ker\varphi\ge r\}
$
corresponds to the tail event
\begin{equation*}
 \left.\ssm(D^{\ge r}_{k,n})\right|_{\beta=0,\,\alpha_i=-x_i}
 =\PP\bigl(h_{n-k}(\Lambda)\ge r\bigr).
\end{equation*}
Thus for ordinary matrices, exact corank gives a probability mass and a closed
degeneracy locus gives the corresponding tail probability.

\begin{remark}[Relation with integrable probability]\label{rem:integrable-prob}
The local weights $(1,1,p,p,1-p,1-p)$ are the symmetric stochastic six-vertex
point, equivalently the rational/XXX degeneration of higher-spin or spin
Hall--Littlewood models \cite{BP}. Rational $R$-matrices also occur naturally in
the SSM and motivic Segre calculus of \cite{KZJ}. Our emphasis here is different:
we fix the family $\{\st_\lambda\}$ and use \eqref{eq:sumone-intro} to obtain a
random partition.

There is, however, an important caveat to this dictionary. The $q$-moment formulas
that make the trigonometric stochastic six-vertex model so computationally powerful
degenerate exactly at the rational/symmetric point used here, since the relevant
fugacity ratio tends to $1$ under the rational limit.
The process itself remains
perfectly well defined, and \cref{thm:intro-prob} is unaffected, but the standard
$q$-Laplace transform observables collapse in this limit and would need to be
replaced by other exact formulas, or recovered as suitable $q\to1$ derivatives,
before they can be brought to bear on the random partition $\Lambda$.
\end{remark}

\section{Skew-symmetric matrices: Maya dimers and corank probabilities}\label{sec:skew}

\subsection{The orbit stratification}
Let $\GL_m$ act on $\Lambda^2\CC^m$ in the usual way, and let
$u_1,\ldots,u_m$ denote the Chern roots of the standard representation. For
$r\equiv m\pmod2$, let $\Sigma^\wedge_{m,r}$ be the orbit of skew-symmetric forms
of corank $r$. In \cite{PR} the authors give explicit, but rather complicated localization-type formulas
$W^\wedge_{m,r}$ for the equivariant CSM classes of these orbits. Consequently,
\[
 \ssm(\Sigma^\wedge_{m,r})
 =\frac{W^\wedge_{m,r}}{c(\Lambda^2\CC^m)},
 \qquad
 c(\Lambda^2\CC^m)=\prod_{i<j}(1+u_i+u_j).
\]
Our goal is to expand these classes in the stable $\st$-basis. The answer is
considerably more rigid than a general expansion: every coefficient is $0$ or $1$,
and membership is governed by a single elementary statistic of the Maya diagram.
This will also give the promised probability interpretation of skew corank.

\subsection{Maya dimers}
For a partition $\lambda$, define its Maya set and occupancy variables by
\[
 B_\lambda=\{\lambda_i-i:i\ge1\}\subset\ZZ,
 \qquad
 w_j(\lambda)=\one_{j\in B_\lambda}.
\]
Thus the doubly infinite binary string $(w_j)_{j\in\ZZ}$ is eventually $1$ to the
left and eventually $0$ to the right.  The string for the empty partition
$\lambda=\varnothing=(0,0,0,\ldots)$, namely $\one_{\ZZ_{<0}}$, is called the
vacuum.  Every Maya string $(w_j)_{j\in\ZZ}$ is a finite perturbation of the
vacuum.

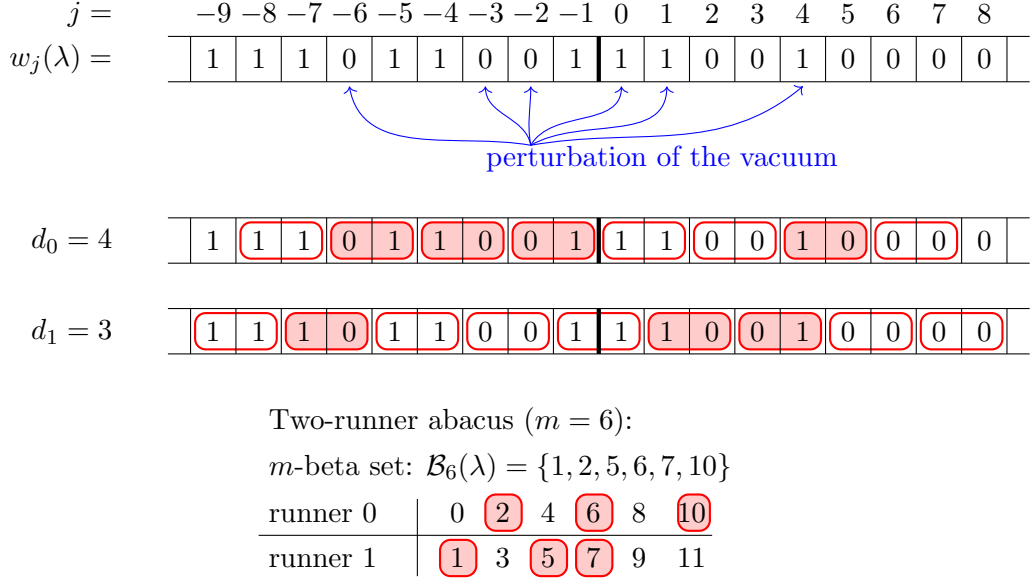
\begin{figure}[H]
\[
\begin{tikzpicture}[scale=.6]
 \pgfmathsetmacro{\y}{0};
 \pgfmathsetmacro{\x}{-10};                         \draw[]  (\x+0.5,\y-.5) -- (\x+0.5,\y+.5);
 \pgfmathsetmacro{\x}{-9}; \node at (\x,\y) {$1$};  \draw[]  (\x+0.5,\y-.5) -- (\x+0.5,\y+.5);
 \pgfmathsetmacro{\x}{-8}; \node at (\x,\y) {$1$};  \draw[]  (\x+0.5,\y-.5) -- (\x+0.5,\y+.5);
 \pgfmathsetmacro{\x}{-7}; \node at (\x,\y) {$1$};  \draw[]  (\x+0.5,\y-.5) -- (\x+0.5,\y+.5);
 \pgfmathsetmacro{\x}{-6}; \node at (\x,\y) {$0$};  \draw[]  (\x+0.5,\y-.5) -- (\x+0.5,\y+.5);
 \pgfmathsetmacro{\x}{-5}; \node at (\x,\y) {$1$};  \draw[]  (\x+0.5,\y-.5) -- (\x+0.5,\y+.5);
 \pgfmathsetmacro{\x}{-4}; \node at (\x,\y) {$1$};  \draw[]  (\x+0.5,\y-.5) -- (\x+0.5,\y+.5);
 \pgfmathsetmacro{\x}{-3}; \node at (\x,\y) {$0$};  \draw[]  (\x+0.5,\y-.5) -- (\x+0.5,\y+.5);
 \pgfmathsetmacro{\x}{-2}; \node at (\x,\y) {$0$};  \draw[]  (\x+0.5,\y-.5) -- (\x+0.5,\y+.5);
 \pgfmathsetmacro{\x}{-1}; \node at (\x,\y) {$1$};  \draw[ultra thick]  (\x+0.5,\y-.5) -- (\x+0.5,\y+.5);
 \pgfmathsetmacro{\x}{0}; \node at (\x,\y) {$1$};  \draw[]  (\x+0.5,\y-.5) -- (\x+0.5,\y+.5);
 \pgfmathsetmacro{\x}{1}; \node at (\x,\y) {$1$};  \draw[]  (\x+0.5,\y-.5) -- (\x+0.5,\y+.5);
 \pgfmathsetmacro{\x}{2}; \node at (\x,\y) {$0$};  \draw[]  (\x+0.5,\y-.5) -- (\x+0.5,\y+.5);
 \pgfmathsetmacro{\x}{3}; \node at (\x,\y) {$0$};  \draw[]  (\x+0.5,\y-.5) -- (\x+0.5,\y+.5);
 \pgfmathsetmacro{\x}{4}; \node at (\x,\y) {$1$};  \draw[]  (\x+0.5,\y-.5) -- (\x+0.5,\y+.5);
 \pgfmathsetmacro{\x}{5}; \node at (\x,\y) {$0$};  \draw[]  (\x+0.5,\y-.5) -- (\x+0.5,\y+.5);
 \pgfmathsetmacro{\x}{6}; \node at (\x,\y) {$0$};  \draw[]  (\x+0.5,\y-.5) -- (\x+0.5,\y+.5);
 \pgfmathsetmacro{\x}{7}; \node at (\x,\y) {$0$};  \draw[]  (\x+0.5,\y-.5) -- (\x+0.5,\y+.5);
 \pgfmathsetmacro{\x}{8}; \node at (\x,\y) {$0$};  \draw[]  (\x+0.5,\y-.5) -- (\x+0.5,\y+.5);
    \draw[]  (-10,.5) -- (9,.5);
    \draw[]  (-10,-.5) -- (9,-.5);
 \node[left] at (-11,\y+1) {$j=$};
 \pgfmathsetmacro{\x}{-9}; \node at (\x,\y+1) {$-9$};
 \pgfmathsetmacro{\x}{-8}; \node at (\x,\y+1) {$-8$};
 \pgfmathsetmacro{\x}{-7}; \node at (\x,\y+1) {$-7$};
 \pgfmathsetmacro{\x}{-6}; \node at (\x,\y+1) {$-6$};
 \pgfmathsetmacro{\x}{-5}; \node at (\x,\y+1) {$-5$};
 \pgfmathsetmacro{\x}{-4}; \node at (\x,\y+1) {$-4$};
 \pgfmathsetmacro{\x}{-3}; \node at (\x,\y+1) {$-3$};
 \pgfmathsetmacro{\x}{-2}; \node at (\x,\y+1) {$-2$};
 \pgfmathsetmacro{\x}{-1}; \node at (\x,\y+1) {$-1$};
 \pgfmathsetmacro{\x}{0}; \node at (\x,\y+1) {$0$};
 \pgfmathsetmacro{\x}{1}; \node at (\x,\y+1) {$1$};
 \pgfmathsetmacro{\x}{2}; \node at (\x,\y+1) {$2$};
 \pgfmathsetmacro{\x}{3}; \node at (\x,\y+1) {$3$};
 \pgfmathsetmacro{\x}{4}; \node at (\x,\y+1) {$4$};
  \pgfmathsetmacro{\x}{5}; \node at (\x,\y+1) {$5$};
  \pgfmathsetmacro{\x}{6}; \node at (\x,\y+1) {$6$};
  \pgfmathsetmacro{\x}{7}; \node at (\x,\y+1) {$7$};
  \pgfmathsetmacro{\x}{8}; \node at (\x,\y+1) {$8$};
 \node[left] at (-11,\y) {$w_j(\lambda)=$};
 \node[blue,left] at (5,-2.2) {perturbation of the vacuum};
 \draw[blue,->] (-2,-1.9) to[out=170,in=-80] (-6,-0.6);
 \draw[blue,->] (-2,-1.9) to[out=100,in=-90] (-3,-0.6);
 \draw[blue,->] (-2,-1.9) to[out=100,in=-90] (-2,-0.6);
 \draw[blue,->] (-2,-1.9) to[out=60,in=-90] (0,-0.6);
 \draw[blue,->] (-2,-1.9) to[out=30,in=-90] (1,-0.6);
 \draw[blue,->] (-2,-1.9) to[out=10,in=-140] (4,-0.6);
 \pgfmathsetmacro{\y}{-4};
\draw[draw=red,thick,rounded corners=4pt]              (-8.4,\y-.4) rectangle (-6.6,\y+.4);
 \draw[draw=red,thick,rounded corners=4pt,fill=red!20] (-6.4,\y-.4) rectangle (-4.6,\y+.4);
 \draw[draw=red,thick,rounded corners=4pt,fill=red!20] (-4.4,\y-.4) rectangle (-2.6,\y+.4);
 \draw[draw=red,thick,rounded corners=4pt,fill=red!20] (-2.4,\y-.4) rectangle (-.6,\y+.4);
 \draw[draw=red,thick,rounded corners=4pt]             (-.4,\y-.4) rectangle (1.4,\y+.4);
 \draw[draw=red,thick,rounded corners=4pt]             (1.6,\y-.4) rectangle (3.4,\y+.4);
 \draw[draw=red,thick,rounded corners=4pt,fill=red!20] (3.6,\y-.4) rectangle (5.4,\y+.4);
 \draw[draw=red,thick,rounded corners=4pt]             (5.6,\y-.4) rectangle (7.4,\y+.4);
 \pgfmathsetmacro{\x}{-10};                         \draw[]  (\x+0.5,\y-.5) -- (\x+0.5,\y+.5);
 \pgfmathsetmacro{\x}{-9}; \node at (\x,\y) {$1$};  \draw[]  (\x+0.5,\y-.5) -- (\x+0.5,\y+.5);
 \pgfmathsetmacro{\x}{-8}; \node at (\x,\y) {$1$};  \draw[]  (\x+0.5,\y-.5) -- (\x+0.5,\y+.5);
 \pgfmathsetmacro{\x}{-7}; \node at (\x,\y) {$1$};  \draw[]  (\x+0.5,\y-.5) -- (\x+0.5,\y+.5);
 \pgfmathsetmacro{\x}{-6}; \node at (\x,\y) {$0$};  \draw[]  (\x+0.5,\y-.5) -- (\x+0.5,\y+.5);
 \pgfmathsetmacro{\x}{-5}; \node at (\x,\y) {$1$};  \draw[]  (\x+0.5,\y-.5) -- (\x+0.5,\y+.5);
 \pgfmathsetmacro{\x}{-4}; \node at (\x,\y) {$1$};  \draw[]  (\x+0.5,\y-.5) -- (\x+0.5,\y+.5);
 \pgfmathsetmacro{\x}{-3}; \node at (\x,\y) {$0$};  \draw[]  (\x+0.5,\y-.5) -- (\x+0.5,\y+.5);
 \pgfmathsetmacro{\x}{-2}; \node at (\x,\y) {$0$};  \draw[]  (\x+0.5,\y-.5) -- (\x+0.5,\y+.5);
 \pgfmathsetmacro{\x}{-1}; \node at (\x,\y) {$1$};  \draw[ultra thick]  (\x+0.5,\y-.5) -- (\x+0.5,\y+.5);
 \pgfmathsetmacro{\x}{0}; \node at (\x,\y) {$1$};  \draw[]  (\x+0.5,\y-.5) -- (\x+0.5,\y+.5);
 \pgfmathsetmacro{\x}{1}; \node at (\x,\y) {$1$};  \draw[]  (\x+0.5,\y-.5) -- (\x+0.5,\y+.5);
 \pgfmathsetmacro{\x}{2}; \node at (\x,\y) {$0$};  \draw[]  (\x+0.5,\y-.5) -- (\x+0.5,\y+.5);
 \pgfmathsetmacro{\x}{3}; \node at (\x,\y) {$0$};  \draw[]  (\x+0.5,\y-.5) -- (\x+0.5,\y+.5);
 \pgfmathsetmacro{\x}{4}; \node at (\x,\y) {$1$};  \draw[]  (\x+0.5,\y-.5) -- (\x+0.5,\y+.5);
 \pgfmathsetmacro{\x}{5}; \node at (\x,\y) {$0$};  \draw[]  (\x+0.5,\y-.5) -- (\x+0.5,\y+.5);
 \pgfmathsetmacro{\x}{6}; \node at (\x,\y) {$0$};  \draw[]  (\x+0.5,\y-.5) -- (\x+0.5,\y+.5);
 \pgfmathsetmacro{\x}{7}; \node at (\x,\y) {$0$};  \draw[]  (\x+0.5,\y-.5) -- (\x+0.5,\y+.5);
 \pgfmathsetmacro{\x}{8}; \node at (\x,\y) {$0$};  \draw[]  (\x+0.5,\y-.5) -- (\x+0.5,\y+.5);
    \draw[]  (-10,\y+.5) -- (9,\y+.5);
    \draw[]  (-10,\y+-.5) -- (9,\y+-.5);
 \node[left] at (-11,\y) {$d_0=4$};
  \pgfmathsetmacro{\y}{-6};
 \draw[draw=red,thick,rounded corners=4pt]             (-9.4,\y-.4) rectangle (-7.6,\y+.4);
 \draw[draw=red,thick,rounded corners=4pt,fill=red!20]  (-7.4,\y-.4) rectangle (-5.6,\y+.4);
 \draw[draw=red,thick,rounded corners=4pt]             (-5.4,\y-.4) rectangle (-3.6,\y+.4);
 \draw[draw=red,thick,rounded corners=4pt]             (-3.4,\y-.4) rectangle (-1.6,\y+.4);
 \draw[draw=red,thick,rounded corners=4pt]             (-1.4,\y-.4) rectangle (.4,\y+.4);
 \draw[draw=red,thick,rounded corners=4pt,fill=red!20] (.6,\y-.4) rectangle (2.4,\y+.4);
 \draw[draw=red,thick,rounded corners=4pt,fill=red!20] (2.6,\y-.4) rectangle (4.4,\y+.4);
 \draw[draw=red,thick,rounded corners=4pt]             (4.6,\y-.4) rectangle (6.4,\y+.4);
 \draw[draw=red,thick,rounded corners=4pt]             (6.6,\y-.4) rectangle (8.4,\y+.4);
 \pgfmathsetmacro{\x}{-10};                         \draw[]  (\x+0.5,\y-.5) -- (\x+0.5,\y+.5);
 \pgfmathsetmacro{\x}{-9}; \node at (\x,\y) {$1$};  \draw[]  (\x+0.5,\y-.5) -- (\x+0.5,\y+.5);
 \pgfmathsetmacro{\x}{-8}; \node at (\x,\y) {$1$};  \draw[]  (\x+0.5,\y-.5) -- (\x+0.5,\y+.5);
 \pgfmathsetmacro{\x}{-7}; \node at (\x,\y) {$1$};  \draw[]  (\x+0.5,\y-.5) -- (\x+0.5,\y+.5);
 \pgfmathsetmacro{\x}{-6}; \node at (\x,\y) {$0$};  \draw[]  (\x+0.5,\y-.5) -- (\x+0.5,\y+.5);
 \pgfmathsetmacro{\x}{-5}; \node at (\x,\y) {$1$};  \draw[]  (\x+0.5,\y-.5) -- (\x+0.5,\y+.5);
 \pgfmathsetmacro{\x}{-4}; \node at (\x,\y) {$1$};  \draw[]  (\x+0.5,\y-.5) -- (\x+0.5,\y+.5);
 \pgfmathsetmacro{\x}{-3}; \node at (\x,\y) {$0$};  \draw[]  (\x+0.5,\y-.5) -- (\x+0.5,\y+.5);
 \pgfmathsetmacro{\x}{-2}; \node at (\x,\y) {$0$};  \draw[]  (\x+0.5,\y-.5) -- (\x+0.5,\y+.5);
 \pgfmathsetmacro{\x}{-1}; \node at (\x,\y) {$1$};  \draw[ultra thick]  (\x+0.5,\y-.5) -- (\x+0.5,\y+.5);
 \pgfmathsetmacro{\x}{0}; \node at (\x,\y) {$1$};  \draw[]  (\x+0.5,\y-.5) -- (\x+0.5,\y+.5);
 \pgfmathsetmacro{\x}{1}; \node at (\x,\y) {$1$};  \draw[]  (\x+0.5,\y-.5) -- (\x+0.5,\y+.5);
 \pgfmathsetmacro{\x}{2}; \node at (\x,\y) {$0$};  \draw[]  (\x+0.5,\y-.5) -- (\x+0.5,\y+.5);
 \pgfmathsetmacro{\x}{3}; \node at (\x,\y) {$0$};  \draw[]  (\x+0.5,\y-.5) -- (\x+0.5,\y+.5);
 \pgfmathsetmacro{\x}{4}; \node at (\x,\y) {$1$};  \draw[]  (\x+0.5,\y-.5) -- (\x+0.5,\y+.5);
 \pgfmathsetmacro{\x}{5}; \node at (\x,\y) {$0$};  \draw[]  (\x+0.5,\y-.5) -- (\x+0.5,\y+.5);
 \pgfmathsetmacro{\x}{6}; \node at (\x,\y) {$0$};  \draw[]  (\x+0.5,\y-.5) -- (\x+0.5,\y+.5);
 \pgfmathsetmacro{\x}{7}; \node at (\x,\y) {$0$};  \draw[]  (\x+0.5,\y-.5) -- (\x+0.5,\y+.5);
 \pgfmathsetmacro{\x}{8}; \node at (\x,\y) {$0$};  \draw[]  (\x+0.5,\y-.5) -- (\x+0.5,\y+.5);
    \draw[]  (-10,\y+.5) -- (9,\y+.5);
    \draw[]  (-10,\y+-.5) -- (9,\y+-.5);
 \node[left] at (-11,\y) {$d_1=3$};
 \pgfmathsetmacro{\xs}{5};
 \pgfmathsetmacro{\y}{-8};
 \node[right] at (-13+\xs,\y) {Two-runner abacus ($m=6$):};
 \node[right] at (-13+\xs,\y-1)
 {$m$-beta set: $\mathcal B_6(\lambda)=\{1,2,5,6,7,10\}$};
 \pgfmathsetmacro{\y}{-10};
\draw[draw=red,thick,rounded corners=4pt,fill=red!20]  (-8+\xs,\y-.4) rectangle (-7.2+\xs,\y+.4);
 \draw[draw=red,thick,rounded corners=4pt,fill=red!20]  (-6+\xs,\y-.4) rectangle (-5.2+\xs,\y+.4);
 \draw[draw=red,thick,rounded corners=4pt,fill=red!20]  (-3.75+\xs,\y-.4) rectangle (-3.05+\xs,\y+.4);
 \draw[draw=red,thick,rounded corners=4pt,fill=red!20]  (-9+\xs,\y-1.4) rectangle (-8.2+\xs,\y-.6);
 \draw[draw=red,thick,rounded corners=4pt,fill=red!20]  (-7+\xs,\y-1.4) rectangle (-6.2+\xs,\y-.6);
 \draw[draw=red,thick,rounded corners=4pt,fill=red!20]  (-6+\xs,\y-1.4) rectangle (-5.2+\xs,\y-.6);
 \node[right] at (-13+\xs,\y) {runner 0};
 \node[right] at (-13+\xs,\y-1) {runner 1};
 \draw[]  (-13+\xs,\y-.5) -- (-3+\xs,\y-.5);
 \draw[]  (-9.5+\xs,\y+.3) -- (-9.5+\xs,\y-1.4);
 \node[right] at (-9+\xs,\y) {0};
 \node[right] at (-8+\xs,\y) {2};
 \node[right] at (-7+\xs,\y) {4};
 \node[right] at (-6+\xs,\y) {6};
 \node[right] at (-5+\xs,\y) {8};
 \node[right] at (-4+\xs,\y) {10};
 \node[right] at (-9+\xs,\y-1) {1};
 \node[right] at (-8+\xs,\y-1) {3};
 \node[right] at (-7+\xs,\y-1) {5};
 \node[right] at (-6+\xs,\y-1) {7};
 \node[right] at (-5+\xs,\y-1) {9};
 \node[right] at (-4+\xs,\y-1) {11};
\end{tikzpicture}
\]
\caption{The Maya string and its two dimerizations for
$\lambda=(5,3,3,3,1,1)$, whose Maya set is
$B_\lambda=\{4,1,0,-1,-4,-5,-7,-8,-9,\ldots\}$.  The lower panel shows the
two-runner abacus for the same partition after choosing $m=6$.  Its $m$-beta set
is
$\mathcal B_6(\lambda)=(B_\lambda+6)\cap\ZZ_{\geq0}
=\{1,2,5,6,7,10\}$; these integers are the bead positions, arranged on runners
$0$ and $1$ according to parity.  Thus the lower panel is an $m$-dependent finite
encoding of the same Maya data.  This encoding is used in \cref{sec:charge} to
define the two-runner charge.}
\label{fig:Maya_dimer}
\end{figure}

We consider the two natural {\em dimerizations} of
$\ZZ$,
\[
 \{2k,2k+1\}\quad\text{and}\quad\{2k-1,2k\},\qquad k\in\ZZ.
\]
For $\epsilon\in\{0,1\}$, let $d_\epsilon(\lambda)$ be the number of dimers in
the corresponding dimerization whose two endpoints have different occupancies:
\[
 d_0(\lambda)=\sum_{k\in\ZZ}|w_{2k}-w_{2k+1}|,\qquad\qquad
 d_1(\lambda)=\sum_{k\in\ZZ}|w_{2k-1}-w_{2k}|.
\]
The sums are finite because the Maya string agrees with the vacuum sufficiently
far in both directions. We may therefore think of $d_\epsilon$ as the number of
``mixed'' dimers, of type $01$ or $10$.
The two statistics have fixed parity:
\begin{equation*}
 d_0(\lambda)\in2\ZZ_{\ge0},\qquad\qquad
 d_1(\lambda)\in2\ZZ_{\ge0}+1.
\end{equation*}
Conjugation of partitions acts on the Maya string by particle-hole reflection,
\[
 w_j(\lambda^T)=1-w_{-j-1}(\lambda).
\]
This reflection preserves each of the two dimerizations and therefore
\begin{equation}\label{eq:dimer-transpose}
 d_\epsilon(\lambda^T)=d_\epsilon(\lambda).
\end{equation}
For $r\equiv\epsilon\pmod2$, the partition of smallest size with
$d_\epsilon(\lambda)=r$ is $\varnothing$ when $r=0$, and for $r\ge1$ it is the
staircase
\[
 \delta_{r-1}=(r-1,r-2,\ldots,1).
\]
In either case its size is $\binom r2$, exactly the codimension of the
skew-symmetric corank-$r$ locus.

There is also a finite $m$-beta-set version of the statistic.  Assume
$\ell(\lambda)\le m$ and define
\[
 \mathcal B_m(\lambda)
 =\{\lambda_i+m-i:1\le i\le m\}
 =(B_\lambda+m)\cap\ZZ_{\geq0}.
\]
The elements of $\mathcal B_m(\lambda)$ are the bead positions on the two-runner
abacus, with the even positions on runner $0$ and the odd positions on runner
$1$.  Write the $m$-beta set in increasing order as
\[
 0\le b_1<\cdots<b_m.
\]
When $\epsilon\equiv m\pmod2$, translating the Maya diagram by $m$ identifies
$d_\epsilon(\lambda)$ with the number of nonnegative pairs
\[
 \{0,1\},\{2,3\},\{4,5\},\ldots
\]
that contain exactly one beta number. In particular,
\[
 d_\epsilon(\lambda)\le m,
 \qquad d_\epsilon(\lambda)\equiv m\pmod2.
\]

\subsection{A refined Littlewood identity}
Define the modified Robbins polynomials as
\begin{equation}\label{eq:robbins}
 \Rstar_B(y;1,1,-2)
 =\frac{
 \ASym\left[
 \prod_{i<j}(1-2y_i+y_iy_j)\prod_i y_i^{b_i}
 \right]
 }{\prod_{i<j}(y_j-y_i)}.
\end{equation}
This is the $w=-2$ specialization of the modified Robbins functions studied in
\cite{FH}.
Our finite-alphabet stable SSM function can be expressed using modified Robbins polynomials as
\begin{equation}\label{eq:robbins-tilde}
 \st_B(u_1,\ldots,u_m)
 =\prod_{i=1}^m(1-y_i)\,
 \Rstar_B(y_1,\ldots,y_m;1,1,-2),
\end{equation}
where $B=(b_1<\cdots<b_m)$ is the beta sequence, and $y_i={u_i}/{(1+u_i)}$.
Define the refined sum
\[
 \Psi_m(t;u)
 =\sum_{\ell(\lambda)\le m}
 t^{d_{m\bmod2}(\lambda)}\st_\lambda(u).
\]

\begin{theorem}[Dimer-refined Littlewood identity]\label{thm:skew-littlewood}
Let $M=M(t;u)$ be the skew-symmetric $m\times m$ matrix with entries
\begin{equation}\label{eq:skew-M}
 M_{ij}
 =\frac{(u_i-u_j)\bigl(1+t^2(u_i+u_j)\bigr)}
        {(u_i+u_j)(1+u_i+u_j)},
 \qquad 1\le i,j\le m.
\end{equation}
For even $m$,
\begin{equation}\label{eq:skew-pf-even}
 \Psi_m(t;u)
 =\prod_{i<j}\frac{u_i+u_j}{u_i-u_j}\,\Pf(M).
\end{equation}
For odd $m$, the same formula holds with $M$ replaced by the skew-symmetric
$(m+1)\times(m+1)$ matrix obtained from $M$ by adjoining an auxiliary index $0$
with
\[
 M_{0j}=t=-M_{j0}\qquad(1\le j\le m).
\]
\end{theorem}

We prove the theorem by a two-component version of the Izergin--Korepin
interpolation used by Fischer--H\"ongesberg in their proof of the unrefined
Littlewood identity \cite[pp.~679--681]{FH}. The point of the method is to
characterize the normalized refined sums uniquely from a degree bound and a small
set of boundary specializations; the Pfaffian expressions are then proved by
checking that they satisfy the same characterization.

For the interpolation argument we work throughout with the variables
\[
 y_i=\frac{u_i}{1+u_i},\qquad Y=(y_1,\ldots,y_n)
\]
already used in \eqref{eq:robbins}--\eqref{eq:robbins-tilde}; the number $n$ of
variables varies during the induction.
For a strict beta set $B=\{0\le b_1<\cdots<b_n\}$ define the two finite dimer
defects
\begin{align*}
 e_0(B)&=\#\{\{0,1\},\{2,3\},\ldots\text{ containing exactly one element of }B\},\\
 e_1(B)&=\one_{0\notin B}
 +\#\{\{1,2\},\{3,4\},\ldots\text{ containing exactly one element of }B\}.
\end{align*}
Translation of the Maya diagram by $n$ gives
\begin{equation}\label{eq:finite-infinite-dimer}
 e_0(\mathcal B_n(\lambda))=d_{n\bmod2}(\lambda),
 \qquad
 \mathcal B_n(\lambda)=\{\lambda_i+n-i:1\le i\le n\}.
\end{equation}
The second statistic $e_1$ is auxiliary; it is forced on us because setting the
last interpolation variable to $0$ interchanges the two dimerizations.

Set
\[
 \mathcal Z_n^{(\epsilon)}(t;Y)
 =\sum_{0\le b_1<\cdots<b_n}
 t^{e_\epsilon(B)}\Rstar_B(Y;1,1,-2)
\]
and, following the normalization in \cite{FH}, define
\begin{equation*}
 Z_n^{(\epsilon)}(t;Y)
 =\prod_{i=1}^n(1-y_i)\prod_{i<j}(1-y_iy_j)\,
 \mathcal Z_n^{(\epsilon)}(t;Y).
\end{equation*}
By \eqref{eq:robbins-tilde} and \eqref{eq:finite-infinite-dimer}, the first
component is exactly the normalized series in \cref{thm:skew-littlewood}:
\begin{equation}\label{eq:skew-Z0-Psi}
 Z_n^{(0)}(t;Y)=\prod_{i<j}(1-y_iy_j)\,\Psi_n(t;u).
\end{equation}
We use the initial values
\[
 Z_0^{(0)}=1,\qquad Z_0^{(1)}=t.
\]

Put $P=y_1\cdots y_n$ and $\sigma=1-t^2$, and assemble
\[
 \mathbf Z_n=
 \begin{pmatrix}Z_n^{(0)}\\Z_n^{(1)}\end{pmatrix}.
\]
The first occupied dimer decomposition gives the exact coupled recurrence
\begin{equation}\label{eq:skew-coupled}
 \mathbf Z_n(Y)
 =\frac{1}{1-P^2}
 \sum_{k=1}^n c_k(Y)
 \begin{pmatrix}
 tP&1\\
 1-\sigma P^2&tP
 \end{pmatrix}
 \mathbf Z_{n-1}(Y\setminus y_k),
\end{equation}
where
\begin{equation}\label{eq:ck-common}
 c_k(Y)
 =(1-y_k)
 \prod_{i\ne k}
 \frac{(1-2y_k+y_iy_k)(1-y_iy_k)y_i}{y_i-y_k}.
\end{equation}
Here is the derivation, including the effect of the dimer statistic.  In the
antisymmetrization formula \eqref{eq:robbins}, choose the variable $y_k$ that
carries the smallest beta number $b_1$.  Delete $b_1$, translate the remaining
beta numbers by $-b_1-1$, and delete $y_k$.  The factors involving $y_k$,
together with the normalizing factors in $Z_n^{(\epsilon)}$, combine to
$c_k(Y)$.  This is the first-beta-number step in the modified Robbins recursion
of \cite[pp.~679--681]{FH}.  What remains is a beta set of size $n-1$ in the
variables $Y\setminus y_k$.

For clarity, we record the extra bookkeeping caused by the refinement.  Write
$b_1=2a$ or $2a+1$.  Translating the remaining beta numbers by $-b_1-1$
either interchanges or preserves the two dimerizations, while the dimer
containing $b_1$ contributes one additional factor $t$ precisely when it is
mixed.  After summing over $a\ge0$, the common geometric series is
$\sum_{a\ge0}P^{2a}=(1-P^2)^{-1}$, and the four possible transitions are
\[
\begin{array}{c|cc}
 &e_0\text{ on the smaller beta set}&e_1\text{ on the smaller beta set}\\ \hline
e_0\text{ on }B&tP&1\\
e_1\text{ on }B&1-\sigma P^2&tP.
\end{array}
\]
For the lower-left entry, the two possibilities contribute
$(1-P^2)+t^2P^2=1-\sigma P^2$ before division by $1-P^2$; the other three
entries are read off directly from whether the translation switches runners
and whether the first dimer is mixed.  Multiplying this transition table by the
row-deletion coefficient $c_k(Y)$ and summing over $k$ proves
\eqref{eq:skew-coupled}.

The three boundary specializations are
\begin{align}
 Z_n^{(0)}(Y,0)&=Z_{n-1}^{(1)}(Y),  \qquad\qquad
 Z_n^{(1)}(Y,0)=Z_{n-1}^{(0)}(Y),\label{eq:skew-zero}\\
 Z_n^{(\epsilon)}(Y,1)
 &=t\prod_{i<n}(1-y_i)Z_{n-1}^{(\epsilon)}(Y),\notag\\
 Z_n^{(\epsilon)}(Y,a,a^{-1})
 &=(t^2-1)\frac{(a-1)^2}{a^{n-1}}
 \prod_{k\le n-2}
 (1-2y_k+ay_k)(y_k+a-2ay_k)
 Z_{n-2}^{(\epsilon)}(Y).
 \label{eq:skew-recip}
\end{align}

For completeness, we also describe the two Pfaffian candidates used in the
simultaneous induction. Set
\[
 S_{ij}=y_i+y_j-2y_iy_j,\qquad T_{ij}=1-y_iy_j,
\]
\[
 \Pi_n(Y)=\prod_{i<j}\frac{S_{ij}T_{ij}}{y_i-y_j},
\]
and define skew-symmetric kernels
\begin{align*}
 C^{(0)}_{ij}
 &=\frac{(y_i-y_j)(T_{ij}-\sigma S_{ij})}{S_{ij}T_{ij}},\\
 C^{(1)}_{ij}
 &=C^{(0)}_{ij}+\sigma(y_i-y_j).
\end{align*}
Finally put $g_i=1-\sigma y_i$. If $n$ is even, the candidates are
\begin{equation}\label{eq:skew-candidates-even}
 Z_n^{(0)}=\Pi_n(Y)\Pf(C^{(0)}),\qquad
 Z_n^{(1)}=t\Pi_n(Y)\Pf(C^{(1)}),
\end{equation}
and if $n$ is odd they are
\begin{equation}\label{eq:skew-candidates-odd}
 Z_n^{(0)}=\Pi_n(Y)\Pf
 \begin{pmatrix}0&t\one^T\\-t\one&C^{(0)}\end{pmatrix},\qquad
 Z_n^{(1)}=\Pi_n(Y)\Pf
 \begin{pmatrix}0&g^T\\-g&C^{(1)}\end{pmatrix}.
\end{equation}
Under $y_i=u_i/(1+u_i)$, the first formulas in
\eqref{eq:skew-candidates-even}--\eqref{eq:skew-candidates-odd}, together with
\eqref{eq:skew-Z0-Psi}, are exactly the Pfaffian stated in
\cref{thm:skew-littlewood}.

We now verify the boundary data on both sides of the proposed identity.  At
$y_n=0$, all terms of \eqref{eq:skew-coupled} except $k=n$ vanish, while
$c_n(Y,0)=1$ and $P=0$.  The transition matrix becomes
$\left(\begin{smallmatrix}0&1\\1&0\end{smallmatrix}\right)$, which proves
\eqref{eq:skew-zero} for the refined sums.  In the Pfaffian candidates the same
exchange follows from
\[
 C^{(0)}(a,0)=g(a),\qquad C^{(1)}(a,0)=1:
\]
the row and column belonging to the zero variable become exactly the border of
the other component, and $\Pi_n(Y,0)=\Pi_{n-1}(Y)$.

At $y_n=1$, substitution in the first-beta recursion gives
\[
 c_k(Y,1)=(1-y_k)c_k(Y)\quad(k<n),\qquad c_n(Y,1)=0.
\]
Using \eqref{eq:skew-coupled} one rank lower and collecting the common factors
gives the second line of \eqref{eq:skew-zero}.  On the Pfaffian side one has
\[
 C^{(0)}(a,1)=-t^2,\qquad C^{(1)}(a,1)=-g(a),
\]
and
\[
 \Pi_n(Y,1)=(-1)^{n-1}\prod_{i<n}(1-y_i)\Pi_{n-1}(Y).
\]
Expanding along the last row and column (and, in odd rank, first subtracting
the appropriate multiple of the border) gives
$t\prod_{i<n}(1-y_i)Z_{n-1}^{(\epsilon)}(Y)$, including the displayed sign.

Finally set $y_{n-1}=a$ and $y_n=a^{-1}$.  In the first-beta recursion the
terms with the two distinguished variables combine, and the remaining terms
are the rank-$(n-2)$ recurrence multiplied by
\[
 (t^2-1)\frac{(a-1)^2}{a^{n-1}}
 \prod_{k\le n-2}(1-2y_k+ay_k)(y_k+a-2ay_k).
\]
This proves \eqref{eq:skew-recip} for the refined sums.  For the candidates,
$\Pi_n$ has a simple zero from $T_{n-1,n}=1-y_{n-1}y_n$, whereas the entry
$C^{(\epsilon)}_{n-1,n}$ has a simple pole.  The rank-two update in
$C^{(1)}-C^{(0)}$ is regular at this point, and
\[
 \lim_{\theta\to a^{-1}}(1-a\theta)C^{(\epsilon)}(a,\theta)
 =(t^2-1)(a-a^{-1}).
\]
Consequently only Pfaffian matchings that pair the two distinguished indices
survive.  Multiplying this residue by
$\lim\Pi_n/\Pi_{n-2}$ gives exactly the preceding prefactor and proves
\eqref{eq:skew-recip} for both candidates.

The only apparent pole in \eqref{eq:skew-coupled} that is not already present in
the inductive terms is the factor $1-P^2$.  In fact the sum
$\sum_kc_k(Y)(\cdots)\mathbf Z_{n-1}(Y\setminus y_k)$ occurring in
\eqref{eq:skew-coupled} --- the numerator of the recurrence --- is divisible by
$1-P^2$. One way to see this is the Pfaffian row-operation argument that also
underlies the unrefined proof of \cite{FH}. Introduce the
coefficients
\[
 \alpha_j=\frac{1+y_j}{y_j}
 \prod_{p\ne j}\frac{1-y_jy_p}{y_j-y_p}
\]
and
\[
 b_i=\frac{1-y_i}{y_i}
 \prod_{p\ne i}\frac{1-2y_i+y_iy_p}{y_i+y_p-2y_iy_p}.
\]
On either component $P=\eta$, $\eta=\pm1$, one has the row identity
\begin{equation}\label{eq:skew-row-identity}
 b_i-(-1)^n\sum_j C^{(0)}_{ij}\alpha_j
 =\eta g_i+(-1)^nt^2.
\end{equation}

Identity \eqref{eq:skew-row-identity} is a residue computation of exactly the
type carried out in detail in \cref{lem:four-residue} below, with one pole
fewer; we deduce it from that computation rather than repeating it.  Write
\[
 S(a,\theta)=a+\theta-2a\theta,\qquad
 T(a,\theta)=1-a\theta,\qquad
 C^{(0)}(a,\theta)=\frac{(a-\theta)\bigl(T(a,\theta)-\sigma S(a,\theta)\bigr)}
                        {S(a,\theta)T(a,\theta)},
\]
so that $S_{ij}=S(y_i,y_j)$, $T_{ij}=T(y_i,y_j)$ and
$C^{(0)}_{ij}=C^{(0)}(y_i,y_j)$.  Fix $i$, put $a=y_i$ and
$Y_{\hat i}=Y\setminus\{y_i\}$, and consider
\begin{equation}\label{eq:skew-residue-fn}
 \mathcal R^\wedge_i(\theta)
 =-\frac{T(a,\theta)-\sigma S(a,\theta)}
         {S(a,\theta)\,\theta\,(1-\theta)}
  \prod_{p\in Y_{\hat i}}\frac{1-\theta p}{\theta-p},
\end{equation}
the skew counterpart of the function $\mathcal R_i$ of \cref{lem:four-residue};
the differences are the numerator and the factor $1-\theta$ in place of
$1-\theta^2$.  Exactly as there, the residues of
\eqref{eq:skew-residue-fn} at the poles $\theta=y_j$, $j\ne i$, are
$C^{(0)}_{ij}\alpha_j$, while the remaining finite poles are
\[
 \theta=0,\qquad \theta=1,\qquad \theta_0=\frac{a}{2a-1},
\]
the pole at $\theta=-1$ present in \cref{lem:four-residue} being absent here.
The two kernel degenerations
\[
 C^{(0)}(a,0)=1-\sigma a=g(a),
 \qquad
 C^{(0)}(a,1)=-t^2
\]
and the relation $\prod_{p\in Y_{\hat i}}p=\eta/a$, valid on the component
$P=\eta$, give
\[
 \operatorname*{Res}_{0}\mathcal R^\wedge_i=(-1)^n\eta g_i,
 \qquad
 \operatorname*{Res}_{1}\mathcal R^\wedge_i=t^2,
 \qquad
 \operatorname*{Res}_{\theta_0}\mathcal R^\wedge_i=-(-1)^nb_i,
\]
the last two independently of $\eta$.  Since
$\mathcal R^\wedge_i(\theta)=O(\theta^{-2})$, the residue at infinity vanishes,
and summing all residues yields \eqref{eq:skew-row-identity}.

Write $y=(y_1,\ldots,y_n)^T$ and $\one=(1,\ldots,1)^T$.  For column vectors
$a,b$, set $a\wedge b=ab^T-ba^T$; thus $y\wedge\one$ is the skew-symmetric
matrix whose $(i,j)$-entry is $y_i-y_j$.  In this notation,
$C^{(1)}-C^{(0)}=\sigma(y\wedge\one)$.  Together with identity
\eqref{eq:skew-row-identity}, this turns the border created by the recurrence
into a linear combination of the existing Pfaffian borders. For odd $n$ the
rank-two update is invisible because all relevant borders lie in
$\operatorname{span}\{\one,y\}$; for even $n$ one additionally uses
\[
 \sum_j\alpha_j=1-\eta,
 \qquad
 \sum_jg_j\alpha_j=t^2(1-\eta)
\]
to see that the two bordered Pfaffians cancel. Thus the recurrence numerator
vanishes for $P=1$ and for $P=-1$, proving the required divisibility.

Before this cancellation, the recurrence gives
\[
 \deg_{y_n}\bigl((1-P^2)Z_n^{(0)}\bigr)\le n+1,
 \qquad
 \deg_{y_n}\bigl((1-P^2)Z_n^{(1)}\bigr)\le n+2.
\]
After dividing by $1-P^2$ we obtain
\begin{equation}\label{eq:skew-degree-bounds}
 \deg_{y_n}Z_n^{(0)}\le n-1,
 \qquad
 \deg_{y_n}Z_n^{(1)}\le n.
\end{equation}
This is the Izergin--Korepin uniqueness step: the normalized functions are
characterized by polynomiality, the degree bounds \eqref{eq:skew-degree-bounds},
and the $n+1$ values
\[
 y_n=0,\qquad y_n=1,\qquad y_n=y_i^{-1}\quad(1\le i<n).
\]
These values determine $Z_n^{(1)}$ uniquely and give one redundant consistency
check for $Z_n^{(0)}$.  The boundary computations above show that the candidates
\eqref{eq:skew-candidates-even}--\eqref{eq:skew-candidates-odd} satisfy
\eqref{eq:skew-zero}--\eqref{eq:skew-recip} and the same initial values.
Simultaneous induction on $n$ therefore proves \cref{thm:skew-littlewood}.

\begin{remark}\label{rem:skew-vs-charge}
The skew interpolation above and the charge-refined interpolation of
\cref{sec:charge} run in parallel: they use the same variables $y_i$, the same
row-deletion coefficients $c_k$ of \eqref{eq:ck-common}, the same coefficients $\alpha_j$ and $b_i$, and the same interpolation points
$y_n\in\{0,1\}\cup\{y_i^{-1}:i<n\}$.  The symmetric case is more involved only
at the pole-cancellation step: it requires a four-residue rather than a
three-residue identity, and its Pfaffian border is a genuine deformation in $z$
rather than the constant borders $t\one$ and $g$.
\end{remark}

\subsection{Comparison with formulas from \texorpdfstring{\cite{PR}}{[PR]}}
Let
$D_m^\wedge(u)=\prod_{i<j}(1+u_i+u_j)$.
Suppose first that $m$ is even.
Multiplying \eqref{eq:skew-pf-even} by $D_m^\wedge$ gives
\[
 D_m^\wedge\Psi_m(t;u)=\Gamma_m(u)\Pf(A+t^2B),
\]
where
\[
 \Gamma_m(u)=\prod_{i<j}\frac{(u_i+u_j)(1+u_i+u_j)}{u_i-u_j},
\]
\[
 A_{ij}=\frac{u_i-u_j}{(u_i+u_j)(1+u_i+u_j)},
 \qquad
 B_{ij}=\frac{u_i-u_j}{1+u_i+u_j}.
\]
Expanding the Pfaffian according to the vertices carried by $B$-edges, and using
Schur's Pfaffian identity \cite{Schur1911}
\[
 \Pf\left[\frac{z_i-z_j}{z_i+z_j}\right]_{i,j\in I}
 =\prod_{\substack{i<j\\ i,j\in I}}\frac{z_i-z_j}{z_i+z_j}
 \qquad\qquad
 \text{(for $|I|$ even),}
\]
gives
\begin{equation*}
 D_m^\wedge[t^r]\Psi_m
 =
 \sum_{\substack{I\subset[m]\\|I|=r}}
 W^\wedge_{m-r}(u_{\bar I})
 \prod_{\substack{i<j\\i,j\in I}}(u_i+u_j)
 \prod_{\substack{i\in I\\j\in\bar I}}
 \frac{(u_i+u_j)(1+u_i+u_j)}{u_i-u_j}.
 \nonumber
\end{equation*}

Now let $m$ be odd.  Then \cref{thm:skew-littlewood} instead gives the bordered
form
\[
 D_m^\wedge\Psi_m(t;u)
 =\Gamma_m(u)\,
 \Pf\begin{pmatrix}
 0&t\one^T\\
 -t\one&A+t^2B
 \end{pmatrix},
\]
which we expand according to the vertices carried by the border and by the
$B$-edges.  Every matching pairs the auxiliary index with exactly one variable
index, so every term carries an odd power of $t$, in agreement with the parity
$d_1\in2\ZZ_{\ge0}+1$ of the exponents of $\Psi_m$.  Let $I\subset[m]$ consist
of the index paired with the border together with the vertices carried by
$B$-edges; then $|I|=r$ is odd, the border edge and the $(r-1)/2$ $B$-edges
together contribute $t^{r}$, and the complementary indices, paired through $A$,
produce the full-rank factor $W^\wedge_{m-r}(u_{\bar I})$ with $m-r$ even,
exactly as before.  The internal factor carried by $I$ is now the bordered
Pfaffian
\[
 \Pf\begin{pmatrix}
 0&\one^T\\
 -\one&\bigl[B_{ij}\bigr]_{i,j\in I}
 \end{pmatrix}
 =\prod_{\substack{i<j\\i,j\in I}}\frac{u_i-u_j}{1+u_i+u_j},
\]
the limit of Schur's identity, applied to the variables $1+2u_i$, as the
auxiliary variable tends to infinity; against the factors of $\Gamma_m$
supported on $I$ it again produces
$\prod_{i<j,\,i,j\in I}(u_i+u_j)$.  Finally, since the auxiliary index precedes
all variable indices, the shuffle sign of $\{0\}\cup I$ against $\bar I$ equals
that of $I$ against $\bar I$, and the sign bookkeeping is unchanged.  Hence the
subset expansion displayed above holds verbatim for odd $m$ as well, now over
the subsets with $|I|=r$ odd.

In both cases the right-hand side is precisely the skew-symmetric $W$-function
$W^\wedge_{m,r}$ of Definition~5.1 in \cite{PR}.  By Theorem~5.5 there,
$W^\wedge_{m,r}=\csm(\Sigma^\wedge_{m,r})$.  Hence
\[
 [t^r]\Psi_m=\ssm(\Sigma^\wedge_{m,r}).
\]

\subsection{From finite rank to the stable classes}\label{ss:skew-stabilization}
Combining the last displayed identity with \cref{thm:skew-littlewood} and the
definition of $\Psi_m$, we obtain the finite-rank expansion
\begin{equation}\label{eq:skew-finite-r}
 \ssm(\Sigma^\wedge_{m,r})
 =\sum_{\substack{\ell(\lambda)\le m\\ d_\epsilon(\lambda)=r}}
 \st_\lambda(u_1,\ldots,u_m),
 \qquad
 \epsilon\equiv m\equiv r\pmod2 .
\end{equation}
(Recall that $\Sigma^\wedge_{m,r}$ is defined only for $r\equiv m\pmod2$; both
sides of \eqref{eq:skew-finite-r} vanish when $r>m$, the left one because the
orbit is then empty, the right one because $d_\epsilon(\lambda)\le m$ whenever
$\ell(\lambda)\le m$.)  The coefficients on the right-hand side do not depend on
$m$, and the length restriction $\ell(\lambda)\le m$ is automatic, since
$\st_\lambda(u_1,\ldots,u_m)=0$ otherwise by \eqref{eq:rho-st}.  This is what
makes the passage to the stable limit possible.

\begin{proposition}[Skew stabilization]\label{prop:skew-stab}
Fix $r\ge0$, let $\epsilon\in\{0,1\}$ satisfy $r\equiv\epsilon\pmod2$, and put
\[
 \Theta^\wedge_r=\sum_{d_\epsilon(\lambda)=r}\st_\lambda\ \in\ \hat\Lambda ,
\]
a well-defined element of $\hat\Lambda$ by \cref{ss:stable-limits}.  Then
\[
 \rho_m\bigl(\Theta^\wedge_r\bigr)=\ssm(\Sigma^\wedge_{m,r})
 \qquad\text{for every } m\equiv r\!\!\pmod 2 .
\]
In particular $\rho^{m'}_m\bigl(\ssm(\Sigma^\wedge_{m',r})\bigr)
=\ssm(\Sigma^\wedge_{m,r})$ for all such $m\le m'$, and $\Theta^\wedge_r$ is the
stable limit of this family.
\end{proposition}

\begin{proof}
Apply \eqref{eq:rho-of-sum} to $\Theta^\wedge_r$: the result is the right-hand
side of \eqref{eq:skew-finite-r}.  Compatibility and uniqueness of the limit are
the last paragraph of \cref{ss:stable-limits}.
\end{proof}

\begin{definition}\label{def:stable-skew}
For $r\ge0$ we define the stable skew-symmetric corank class by
\[
 \ssm(\Sigma^\wedge_{\infty,r}):=\Theta^\wedge_r\in\hat\Lambda .
\]
\end{definition}

This is the precise meaning of the phrase ``stabilization along dimensions of the
same parity as $r$'' used in \cref{thm:intro-skew}, and with it
\eqref{eq:skew-main-intro} is a restatement of \cref{prop:skew-stab}.  Thus
\cref{thm:intro-skew} is proved.

\begin{remark}
The compatibility in \cref{prop:skew-stab} is not a formal consequence of a
geometric inclusion: the corank-$r$ orbits in $\Lambda^2\CC^m$ and in
$\Lambda^2\CC^{m'}$ are not related by a $\GL_m$-equivariant inclusion of pairs,
and $\rho^{m'}_m$ is merely the restriction of equivariant coefficients along
$\GL_m\hookrightarrow\GL_{m'}$.  It is a genuine property of the classes, and
here we read it off from \eqref{eq:skew-finite-r}.
\end{remark}

\begin{example}
The condition $d_0(\lambda)=0$ forces every even dimer to have occupancy $00$ or
$11$, equivalently
\[
 \lambda=(2\mu_1,2\mu_1,2\mu_2,2\mu_2,\ldots).
\]
At the other extreme, the first partition with $d_\epsilon=r$ is
$\delta_{r-1}$, so the lowest-degree term of \eqref{eq:skew-main-intro} is the
classical fundamental class $s_{r-1,r-2,\ldots,1}$.
\end{example}

\begin{example}[Initial stable skew expansions]\label{ex:skew-expansions}
Here are the first few expansions according to \cref{thm:intro-skew}. We write
$\mathcal O_{\ge d}$ for terms $\st_\lambda$ with $|\lambda|\ge d$.
\begin{align*}
\ssm(\Sigma^\wedge_{\infty,0})
={}&
\st_{\varnothing}
+\st_{22}
+(\st_{44}+\st_{2222})
+(\st_{66}+\st_{4422}+\st_{222222})
+\mathcal O_{\ge16},
\\[2mm]
\ssm(\Sigma^\wedge_{\infty,1})
={}&
\st_{\varnothing}
+\st_{1}
+(\st_{2}+\st_{11})
+(\st_{3}+\st_{111})
+(\st_{4}+\st_{1111})\\
&\quad
+(\st_{5}+\st_{11111})
+(\st_{6}+\st_{33}
+\st_{222}+\st_{111111})
+\mathcal O_{\ge7},
\\[2mm]
\ssm(\Sigma^\wedge_{\infty,2})
={}&
\st_{1}
+(\st_{2}+\st_{11})
+(\st_{3}+\st_{21}+\st_{111})
+(\st_{4}+\st_{31}
+\st_{211}+\st_{1111})
+\mathcal O_{\ge5},
\\[2mm]
\ssm(\Sigma^\wedge_{\infty,3})
={}&
\st_{21}
+(\st_{31}+\st_{22}+\st_{211})
+(\st_{41}+\st_{32}+\st_{311}
+\st_{221}+\st_{2111})
+\mathcal O_{\ge6},
\\[2mm]
\ssm(\Sigma^\wedge_{\infty,4})
={}&
\st_{321}
+(\st_{421}+\st_{331}
+\st_{322}+\st_{3211})\\
&\quad
+(\st_{521}+\st_{431}+\st_{422}
+\st_{4211}
+\st_{332}+\st_{3311}
+\st_{3221}+\st_{32111})
+\mathcal O_{\ge9}.
\end{align*}
The $0/1$ coefficients and transpose symmetry are visible already in these small
terms.
\end{example}

\section{Symmetric matrices: 2-cores and Chebyshev polynomials}\label{sec:symmetric-statement}

\subsection{2-cores}
Let $\GL_m$ act on $S^2\CC^m$, and let $\Sigma^S_{m,r}$ denote the orbit of
symmetric forms of corank~$r$.
Definition~5.3 of \cite{PR} gives explicit localization formulas $W^S_{m,r}$, and Theorem~5.5 there identifies them with the CSM
classes of the symmetric orbits.  Hence
\begin{equation*}
 \ssm(\Sigma^S_{m,r})
 =\frac{W^S_{m,r}}{D_m^S(u)},
 \qquad
 D_m^S(u)=\prod_{i\le j}(1+u_i+u_j).
\end{equation*}
The first stable $\st$-expansions in \cite{PR} display remarkable repeated
coefficients.  The following theorem explains them.

Recall that the $2$-core of a partition is obtained by repeatedly removing rim
dominoes.  Every $2$-core is a staircase $\delta_k$.  We use the notation
\eqref{eq:kappa-def}.

\begin{theorem}[Finite-rank $2$-core/Chebyshev formula]\label{thm:sym-finite}
For every $m$,
\begin{equation}\label{eq:sym-finite}
 \sum_{r=0}^m t^r\ssm(\Sigma^S_{m,r})(u_1,\ldots,u_m)
 =\sum_{\ell(\lambda)\le m}
 V_{\kappa(\lambda)}(t)\st_\lambda(u_1,\ldots,u_m).
\end{equation}
\end{theorem}

The proof of this theorem will occupy Sections~\ref{sec:charge} and \ref{sec:pr-symmetric}.
The stable form of this theorem is \cref{thm:intro-sym}.  The passage from one to
the other is the same as in the skew case, purely formal.

\begin{definition}\label{def:stable-sym}
For $r\ge0$ put
\[
 \Theta^S_r=\sum_\lambda\bigl([t^r]V_{\kappa(\lambda)}(t)\bigr)\,\st_\lambda
 \ \in\ \hat\Lambda ,
\]
a well-defined element of $\hat\Lambda$ by \cref{ss:stable-limits}, and set
$\ssm(\Sigma^S_{\infty,r}):=\Theta^S_r$.
\end{definition}

Indeed, extracting the coefficient of $t^r$ from \eqref{eq:sym-finite} and
applying \eqref{eq:rho-of-sum} gives
\[
 \rho_m\bigl(\Theta^S_r\bigr)=\ssm(\Sigma^S_{m,r})
 \qquad\text{for every }m\ge0,
\]
with the convention $\Sigma^S_{m,r}=\varnothing$ for $r>m$; the case $r>m$ is
consistent because $\deg V_{\kappa(\lambda)}=\kappa(\lambda)\le\ell(\lambda)\le m$
for the partitions occurring in \eqref{eq:sym-finite}.  Hence the classes
$\ssm(\Sigma^S_{m,r})$ are compatible under the maps $\rho^{m'}_m$ and
$\Theta^S_r$ is their stable limit, so that \cref{thm:intro-sym} is exactly
\eqref{eq:sym-finite} transported to $\hat\Lambda$.  In contrast with the skew
case no parity restriction is needed here: the orbits $\Sigma^S_{m,r}$ exist for
all $0\le r\le m$, and the coefficient $[t^r]V_{\kappa(\lambda)}(t)$ does not
depend on $m$.

\subsection{Core blocks and explicit coefficients}
Define
\begin{equation*}
 \mathcal C_k
 =\sum_{\core(\lambda)=\delta_k}\st_\lambda.
\end{equation*}
Then \cref{thm:intro-sym} is the universal Chebyshev transform
\begin{equation}\label{eq:core-transform}
 \sum_{r\ge0}t^r\ssm(\Sigma^S_{\infty,r})
 =\sum_{k\ge0}V_k(t)\mathcal C_k.
\end{equation}
Thus the entire $2$-quotient of $\lambda$ is invisible to the stable SSM
coefficient.
Writing
\[
 a_{k,r}=[t^r]V_k(t),
\]
one obtains the closed formula
\begin{equation*}
 a_{k,r}=
 \begin{cases}
 2^r(-1)^{(k-r)/2}
 \displaystyle\binom{(k+r)/2}{(k-r)/2},
 &k-r\text{ even},\\[4mm]
 2^r(-1)^{(k-r+1)/2}
 \displaystyle\binom{(k+r-1)/2}{(k-r-1)/2},
 &k-r\text{ odd},
 \end{cases}
\end{equation*}
with $a_{k,r}=0$ for $r>k$.
For example,
\[
\begin{array}{c|rrrrr}
 &t^0&t^1&t^2&t^3&t^4\\ \hline
V_0&1&0&0&0&0\\
V_1&-1&2&0&0&0\\
V_2&-1&-2&4&0&0\\
V_3&1&-4&-4&8&0\\
V_4&1&4&-12&-8&16
\end{array}
\]
In particular, $a_{r,r}=2^r$.  Since the smallest partition with $2$-core
$\delta_r$ is $\delta_r$ itself, the lowest-degree term of the corank-$r$ class is
\[
 2^rs_{r,r-1,\ldots,1},
\]
the classical fundamental class of the symmetric corank-$r$ locus \cite{JLP,HT,FNR,AF}.

\begin{example}[Initial stable symmetric expansions]\label{ex:sym-expansions}
The first terms of the Chebyshev transform \eqref{eq:core-transform} are as follows.
\begin{align*}
\ssm(\Sigma^S_{\infty,0})
={}&
\st_{\varnothing}
-\st_{1}
+(\st_{2}+\st_{11})
-(\st_{3}+\st_{21}+\st_{111})\\
&\quad
+(\st_{4}+\st_{31}+\st_{22}
+\st_{211}+\st_{1111})
+\mathcal O_{\ge 5},
\\[2mm]
\ssm(\Sigma^S_{\infty,1})
={}&
2\st_{1}
+2(\st_{3}-\st_{21}+\st_{111})\\
&\quad
+2(\st_{5}-\st_{41}+\st_{32}
+\st_{311}+\st_{221}
-\st_{2111}+\st_{11111})
+\mathcal O_{\ge 6},
\\[2mm]
\ssm(\Sigma^S_{\infty,2})
={}&
4\st_{21}
+4(\st_{41}+\st_{2111})
-4\st_{321}\\
&\quad
+4(\st_{61}+\st_{43}
+\st_{4111}+\st_{2221}
+\st_{211111})
+\mathcal O_{\ge 8},
\\[2mm]
\ssm(\Sigma^S_{\infty,3})
={}&
8\st_{321}
+8(\st_{521}+\st_{32111})\\
&\quad
+8(\st_{721}+\st_{541}
+\st_{52111}+\st_{32221}
+\st_{3211111}
-\st_{4321})
+\mathcal O_{\ge 12},
\\[2mm]
\ssm(\Sigma^S_{\infty,4})
={}&
16\st_{4321}
+16(\st_{6321}+\st_{432111})\\
&\quad
+16(\st_{8321}+\st_{6521}
+\st_{632111}+\st_{432221}
+\st_{43211111})
+\mathcal O_{\ge 15}.
\end{align*}
These examples display both features explained by the theorem: coefficients are
constant on $2$-core blocks, and the diagonal coefficient is $2^r$.
\end{example}

\begin{corollary}[Transpose symmetry]\label{cor:sym-transpose}
The coefficient of $\st_\lambda$ in every symmetric stable SSM class is
equal to the coefficient of $\st_{\lambda^T}$.
\end{corollary}

\begin{proof}
The $2$-core of $\lambda^T$ is the transpose of the $2$-core of $\lambda$, and
staircases are self-transpose.
\end{proof}

Together with the skew-symmetric transpose symmetry \eqref{eq:dimer-transpose},
this proves the transpose-invariance part of \cite[Conj.~6.2]{PR}. The sign assertion of \cite[Conj.~6.2]{PR} concerns the same expansions:
for {\em even} $r$ the coefficients should alternate with the degree.  This
also follows.

\begin{corollary}[Sign alternation]\label{cor:sym-signs}
Let $r$ be even.  Then every nonzero coefficient in the expansion of
$\ssm(\Sigma^S_{\infty,r})$ satisfies
\[
 [\st_\lambda]\ssm(\Sigma^S_{\infty,r})
 \in(-1)^{|\lambda|-|\delta_r|}\ZZ_{>0}.
\]
This proves the sign assertion of \cite[Conj.~6.2]{PR}.
\end{corollary}

\begin{proof}
\cref{thm:intro-sym} determines every coefficient,
$[\st_\lambda]\ssm(\Sigma^S_{\infty,r})=a_{\kappa(\lambda),r}$; by the closed
formula for $a_{k,r}$ its sign is $(-1)^{\lceil(k-r)/2\rceil}$ for $r\le k$,
and the coefficient vanishes for $r>k$.  For
even $r$ one has $\lceil(k-r)/2\rceil=\lceil k/2\rceil-r/2$, and
\[
 \Bigl\lceil \tfrac k2\Bigr\rceil\equiv\tfrac{k(k+1)}2,
 \qquad
\tfrac r2\equiv\tfrac{r(r+1)}2
 \pmod 2 ,
\]
so that this sign equals $(-1)^{|\delta_k|-|\delta_r|}=(-1)^{|\lambda|-|\delta_r|}$
for $\core(\lambda)=\delta_k$.  This is exactly the asserted alternation.
\end{proof}

\begin{remark}\label{rem:sym-signs-odd}
For odd $r$ no such rule can hold: the degree-$3$ part of
$\ssm(\Sigma^S_{\infty,1})$ in \cref{ex:sym-expansions} is
$2(\st_3-\st_{21}+\st_{111})$, whose coefficients have different signs within the same degree.
No positivity statement holds in the symmetric case.  In the language of
\cref{sec:prob-consequences}, this is precisely the difference between an {\em event}
and a {\em signed observable}: skew corank is an event, whereas symmetric corank is a
Chebyshev moment of the random $2$-core, and the Chebyshev polynomials $V_k$ are
not positive.
\end{remark}

\begin{corollary}[Two specializations]\label{cor:t0t1}
One has
\[
 \sum_r\ssm(\Sigma^S_{\infty,r})=1
\]
and
\[
 \ssm(\Sigma^S_{\infty,0})
 =\sum_\lambda(-1)^{|\lambda|}\st_\lambda.
\]
\end{corollary}

\begin{proof}
The first identity follows from $V_k(1)=1$.  For the second, note that
$V_k(0)=(-1)^{k(k+1)/2}$ and
$|\lambda|\equiv|\delta_k|=k(k+1)/2\pmod2$ whenever $\core(\lambda)=\delta_k$.
\end{proof}

\section{The two-runner charge and a refined Littlewood identity}\label{sec:charge}

The purpose of this section is to prove the symmetric-function identity underlying
\cref{thm:sym-finite}.  We first replace the $2$-core by an equivalent two-runner
charge $\Delta_m$ and show that the Chebyshev coefficient is obtained from the
charge monomial $z^{\Delta_m}$ by the elementary symmetrization
\eqref{eq:charge-cheb}--\eqref{eq:S-A}.  We then evaluate the resulting
charge-refined sum $\mathcal A_m(z;u)$ by the Pfaffian formula of
\cref{thm:charge-pf}, using a modified Robbins recurrence and interpolation.
These two outputs are the only ingredients from this section used in
\cref{sec:pr-symmetric}, where the symmetrized charge Pfaffian is identified with
the \cite{PR} corank generating series and the proof of
\cref{thm:sym-finite} is completed.

The $2$-core itself is not the most convenient statistic
for the modified Robbins recurrence.  The natural statistic is a linear charge on
the two-runner abacus.

\subsection{The parity charge}
For $\ell(\lambda)\le m$, let
\begin{equation}\label{eq:finite-beta-sym}
 \mathcal B_m(\lambda)=\{\lambda_i+m-i:1\le i\le m\}
 =
 (B_\lambda+m)\cap \ZZ_{\geq 0},
\end{equation}
be the $m$-beta set of $\lambda$.  Its elements are the bead positions on the
two-runner abacus, with even positions on runner $0$ and odd positions on runner
$1$.  Equivalently, the second equality in \eqref{eq:finite-beta-sym} says that
the $m$-beta set is obtained by translating the Maya set by $m$ and retaining
its nonnegative elements.  The case $m=6$ is illustrated in
\cref{fig:Maya_dimer}.
Define
\begin{equation}\label{eq:Delta}
 \Delta_m(\lambda)
 =\sum_{b\in \mathcal B_m(\lambda)}(-1)^b
 =\#\{b\in\mathcal B_m(\lambda):b\text{ even}\}
  -\#\{b\in\mathcal B_m(\lambda):b\text{ odd}\}.
\end{equation}
Removing a domino from the Young diagram moves one bead two positions along the
same runner, so $\Delta_m$ is invariant under domino removal.

\begin{lemma}[Charge and the $2$-core]\label{lem:charge-core}
If $\core(\lambda)=\delta_k$, then
\begin{equation*}
 \Delta_m(\lambda)=
 \begin{cases}
 -k,&m\equiv k\pmod2,\\
 k+1,&m\not\equiv k\pmod2.
 \end{cases}
\end{equation*}
\end{lemma}

\begin{proof}
By domino invariance it is enough to evaluate \eqref{eq:Delta} on the staircase
$\delta_k$.  Its $m$-beta set is an arithmetic progression of one parity,
together with the vacuum tail; the displayed count follows immediately.
\end{proof}

Equivalently, if the Young diagram is colored as a chessboard with $(1,1)$ of sign
$+1$, define its BG-rank, following Berkovich--Garvan \cite{BG}, by
\[
 \BG(\lambda)=\sum_{(i,j)\in\lambda}(-1)^{i+j}.
\]
Then
\begin{equation*}
 \Delta_m(\lambda)=
 \begin{cases}
 2\BG(\lambda),&m\text{ even},\\
 1-2\BG(\lambda),&m\text{ odd}.
 \end{cases}
\end{equation*}
Thus the BG-rank, the two-runner charge, and the $2$-core encode the same datum.

\subsection{Chebyshev symmetrization}
Choose $z$ with
\begin{equation*}
 2t=z+z^{-1}.
\end{equation*}
The Chebyshev polynomials in \eqref{eq:cheb-rec-intro} satisfy
\begin{equation}\label{eq:cheb-z}
 V_k\!\left(\frac{z+z^{-1}}2\right)
 =\frac{z^{k+1}+z^{-k}}{1+z}.
\end{equation}
Combining \cref{lem:charge-core} with \eqref{eq:cheb-z} gives
\begin{equation}\label{eq:charge-cheb}
 V_{\kappa(\lambda)}(t)
 =\frac{z^{\Delta_m(\lambda)}+z^{1-\Delta_m(\lambda)}}{1+z}.
\end{equation}

Define the charge-refined sum
\begin{equation*}
 \mathcal A_m(z;u)
 =\sum_{\ell(\lambda)\le m}
 z^{\Delta_m(\lambda)}\st_\lambda(u).
\end{equation*}
Then \eqref{eq:charge-cheb} says that the $2$-core/Chebyshev sum is obtained from
$\mathcal A_m$ by the simple involutive symmetrization
\begin{equation}\label{eq:S-op}
 \mathscr S_z f(z)
 :=\frac{f(z)+zf(z^{-1})}{1+z}.
\end{equation}
Indeed,
\begin{equation}\label{eq:S-A}
 \mathscr S_z\mathcal A_m(z;u)
 =\sum_{\ell(\lambda)\le m}
 V_{\kappa(\lambda)}(t)\st_\lambda(u).
\end{equation}

\subsection{The charge recurrence}
We now use \eqref{eq:robbins-tilde}--\eqref{eq:robbins}. As in the skew
interpolation argument, write
\[
 y_i=\frac{u_i}{1+u_i},\qquad Y=(y_1,\ldots,y_n),\qquad P=y_1\cdots y_n,
\]
and use $\mathcal A_n(z;Y)$ as shorthand for $\mathcal A_n(z;u)$ after the
substitution $u_i=y_i/(1-y_i)$. Following the Fischer--H\"ongesberg
normalization, set
\[
 Z_n(z;Y)=\prod_{i<j}(1-y_iy_j)\,\mathcal A_n(z;Y).
\]
The decomposition according to the first beta number gives the exact recurrence
\begin{equation}\label{eq:charge-recurrence}
 Z_n(z;Y)
 =\frac1{1-P^2}
 \sum_{k=1}^n c_k(Y)
 \left[
 zZ_{n-1}(z^{-1};Y_{\hat k})
 +Pz^{-1}Z_{n-1}(z;Y_{\hat k})
 \right],
\end{equation}
where $c_k$ is the same coefficient as in \eqref{eq:ck-common}.
The two states are now $z$ and $z^{-1}$ rather than the two dimerizations of
\cref{sec:skew}.

We spell out the first-beta-number decomposition.  For a beta set
$B=\{b_1<\cdots<b_n\}$, choose in the antisymmetrization
\eqref{eq:robbins} the variable $y_k$ carrying $b_1$, and set
\[
 B'=\{b_2-b_1-1,\ldots,b_n-b_1-1\}.
\]
Then $B'$ is a strict nonnegative beta set of size $n-1$.  The normalized
factors involving $y_k$ give $c_k(Y)$, exactly as in the derivation of
\eqref{eq:skew-coupled}.  If $b_1=2a$, then
\[
 \Delta_n(B)=1-\Delta_{n-1}(B'),
\]
whereas if $b_1=2a+1$, then
\[
 \Delta_n(B)=-1+\Delta_{n-1}(B').
\]
The translation of all $n$ beta exponents supplies $P^{b_1}$.
Consequently the even values of $b_1$ contribute
\[
 \sum_{a\ge0}P^{2a}z^{1-\Delta_{n-1}(B')}
 =\frac{z}{1-P^2}(z^{-1})^{\Delta_{n-1}(B')},
\]
and the odd values contribute
\[
 \sum_{a\ge0}P^{2a+1}z^{-1+\Delta_{n-1}(B')}
 =\frac{Pz^{-1}}{1-P^2}z^{\Delta_{n-1}(B')}.
\]
After summing over $B'$ and over the choice of $k$, these are precisely the two
terms in \eqref{eq:charge-recurrence}.  This proves the recurrence directly
from \eqref{eq:robbins} and also explains why its two states are exchanged by
$z\leftrightarrow z^{-1}$.

\subsection{The Pfaffian candidate}
Write
$q=z+z^{-1}$.
In the $u$-variables define
\begin{equation}\label{eq:R-prefactor}
 R_m(u)
 =\frac1{\prod_i(1+2u_i)}
 \prod_{i<j}\frac{u_i+u_j}{u_i-u_j}
\end{equation}
and
\begin{equation}\label{eq:Cq}
 C_{ij}^{(q)}
 =\frac{(u_i-u_j)
 \bigl(1+u_i+u_j+4u_iu_j+q^2u_iu_j(u_i+u_j)\bigr)}
 {(u_i+u_j)(1+u_i+u_j)}.
\end{equation}

\begin{theorem}[Charge-refined Littlewood identity]\label{thm:charge-pf}
If $m$ is odd, then
\begin{equation*}
 \mathcal A_m(z;u)
 =R_m(u)\,
 \Pf\begin{pmatrix}
 0&(z+qu_j)_j\\
 -(z+qu_i)_i&(C_{ij}^{(q)})
 \end{pmatrix}.
\end{equation*}
If $m$ is even, then
\begin{equation*}
 \mathcal A_m(z;u)
 =R_m(u)\,
 \Pf\begin{pmatrix}
 0&1&\one^T\\
 -1&0&((1+z^2)u)^T\\
 -\one&-(1+z^2)u&C^{(q)}
 \end{pmatrix}.
\end{equation*}
\end{theorem}

The rest of the section proves \cref{thm:charge-pf}.  We give the interpolation
argument because it is the new technical ingredient in the symmetric case.

\subsection{The candidate in the \texorpdfstring{$y$}{y}-variables}
\label{ss:charge-candidate}
After the substitution $u_i=y_i/(1-y_i)$, multiplying the charge Pfaffian by the
normalizing factor in the definition of $Z_n$ gives the prefactor
\begin{equation}\label{eq:Pibar}
 \overline\Pi_n(Y)
 =\prod_{i<j}
 \frac{(y_i+y_j-2y_iy_j)(1-y_iy_j)}{y_i-y_j}.
\end{equation}
The factor absorbed into each variable row and column is
\[
 L(a)=\frac{1-a}{1+a},
 \qquad
 H(a)=\frac{a}{1+a},
 \qquad
 L_i=L(y_i),\quad H_i=H(y_i).
\]
Write $S(a,\theta)=a+\theta-2a\theta$, as in \cref{sec:skew}.  The resulting
bulk kernel is
\begin{equation}\label{eq:Kbar}
 \overline K(a,\theta)
 =\frac{(a-\theta)N(a,\theta)}{(1+a)(1+\theta)S(a,\theta)(1-a\theta)},
 \qquad
 \overline K_{ij}=\overline K(y_i,y_j),
\end{equation}
where
\begin{equation}\label{eq:N-explicit}
 N(a,\theta)=(1+3a\theta)(1-a)(1-\theta)+q^2a\theta\,S(a,\theta).
\end{equation}
For later reference, the scaled odd-rank boundary is
\[
 G_i(z)=zL_i+qH_i=\frac{z+z^{-1}y_i}{1+y_i}.
\]
Thus the normalized Pfaffian candidate can be written as
\begin{equation}\label{eq:charge-cand-odd}
 Z_n=\overline\Pi_n
 \Pf\begin{pmatrix}0&G(z)^T\\-G(z)&\overline K\end{pmatrix}
 \qquad (n\ \text{odd}),
\end{equation}
and
\begin{equation}\label{eq:charge-cand-even}
 Z_n=\overline\Pi_n
 \Pf\begin{pmatrix}
 0&1&L^T\\
 -1&0&(zqH)^T\\
 -L&-zqH&\overline K
 \end{pmatrix}
 \qquad (n\ \text{even}),
\end{equation}
where we used $1+z^2=zq$.

\subsection{Boundary values}\label{ss:charge-boundary}
The Pfaffian candidate, as well as the recurrence \eqref{eq:charge-recurrence},
satisfies
\begin{equation}\label{eq:charge-zero}
 Z_n(z;Y,0)=zZ_{n-1}(z^{-1};Y),
\end{equation}
\begin{equation}\label{eq:charge-one}
 Z_n(z;Y,1)=\frac q2\prod_{i<n}(1-y_i)Z_{n-1}(z;Y),
\end{equation}
and, for $y_{n-1}=a$, $y_n=a^{-1}$,
\begin{align}\label{eq:charge-recip}
 Z_n(z;Y,a,a^{-1})
 ={}&(q^2-4)
 \frac{(a-1)^2}{(a+1)^2a^{n-2}}\\
 &\times
 \prod_{i\le n-2}
 (1-2y_i+ay_i)(y_i+a-2ay_i)
 Z_{n-2}(z;Y).
 \notag
\end{align}
These are the same interpolation locations as in the skew problem.  We indicate
the verifications.

On the side of the sum $\mathcal A_n$, identity \eqref{eq:charge-zero} is the
statement that increasing the number of variables shifts the beta set: if
$\ell(\lambda)\le n-1$, then
\[
 \mathcal B_n(\lambda)=\bigl(\mathcal B_{n-1}(\lambda)+1\bigr)\cup\{0\},
 \qquad\text{whence}\qquad
 \Delta_n(\lambda)=1-\Delta_{n-1}(\lambda)
\]
by \eqref{eq:Delta}.  Since $\st_\lambda(u_1,\ldots,u_{n-1},0)
=\st_\lambda(u_1,\ldots,u_{n-1})$, and since this vanishes when
$\ell(\lambda)>n-1$, setting $y_n=0$ replaces $z^{\Delta_n(\lambda)}$ by
$z\cdot(z^{-1})^{\Delta_{n-1}(\lambda)}$ termwise; the normalizing product
$\prod_{i<j}(1-y_iy_j)$ is unaffected.  This is the precise sense in which the
last variable interchanges the two runners.  The other two specializations of
$\mathcal A_n$ are the specializations of the modified Robbins recurrence
\eqref{eq:charge-recurrence} at the corresponding points, and are the same
computations as in \cite[pp.~679--681]{FH}.

For the Pfaffian candidate
\eqref{eq:charge-cand-odd}--\eqref{eq:charge-cand-even}, all three identities
come from the degenerations
\begin{equation}\label{eq:kernel-degenerations}
 \overline K(a,0)=L(a),
 \quad
 \overline K(a,1)=-\frac{q^2}{2}H(a),
 \quad
 \lim_{\theta\to a^{-1}}(1-a\theta)\overline K(a,\theta)
 =(q^2-4)\frac{a-1}{a+1},
\end{equation}
all three immediate from \eqref{eq:Kbar}--\eqref{eq:N-explicit}; for the last
one note that $N(a,a^{-1})=(q^2-4)S(a,a^{-1})$.  We describe the mechanism for
odd $n$; for even $n$ it is the same computation with the two forms
\eqref{eq:charge-cand-odd} and \eqref{eq:charge-cand-even} exchanged.

At $y_n=0$ one has $L_n=1$, $H_n=0$, $G_n(z)=z$ and, by \eqref{eq:Pibar},
$\overline\Pi_n(Y,0)=\overline\Pi_{n-1}(Y)$.  By the first degeneration in
\eqref{eq:kernel-degenerations} the kernel row of the vanishing variable becomes
the constant vector $L$.  Subtracting $z$ times that row and column from the
$G(z)$-border replaces $G(z)=zL+qH$ by $qH$; rescaling the border row and column
by $z^{-1}$ then turns the bordered Pfaffian into the one in
\eqref{eq:charge-cand-even} with $z$ replaced by $z^{-1}$, at the cost of the
factor $z$.  This is \eqref{eq:charge-zero}.

At $y_n=1$ one has $L_n=0$, $H_n=\frac12$, $G_n(z)=\frac q2$ and
$\overline\Pi_n(Y,1)=(-1)^{n-1}\prod_{i<n}(1-y_i)\,\overline\Pi_{n-1}(Y)$.  By
the second degeneration the kernel row of the last variable is
$-\frac{q^2}2H$; factoring $-\frac q2$ out of that row and column, and then
performing the same single row operation as above, produces
\eqref{eq:charge-cand-even}, and routine sign bookkeeping gives
\eqref{eq:charge-one}.

Finally, at $y_{n-1}=a$, $y_n=a^{-1}$ the prefactor $\overline\Pi_n$ acquires a
simple zero from the factor $1-y_{n-1}y_n$, while $\overline K_{n-1,n}$ acquires
a simple pole there.  Consequently only those terms of the Pfaffian expansion in
which the indices $n-1$ and $n$ are paired with each other survive the limit;
what remains is $Z_{n-2}$, multiplied by the residue in the third degeneration
of \eqref{eq:kernel-degenerations} and by the limit of
$\overline\Pi_n/\overline\Pi_{n-2}$ at $y_{n-1}=a$, $y_n=a^{-1}$.  Together these
give the prefactor in \eqref{eq:charge-recip}; in particular the factor
$q^2-4=(z-z^{-1})^2$ there is that residue.

\subsection{Pole cancellation by a four-residue identity}
The only non-obvious issue in \eqref{eq:charge-recurrence} is the apparent pole at
$P^2=1$.  We record the row identity that removes it.

Define
\begin{equation*}
 b_i=\frac{1-y_i}{y_i}
 \prod_{p\ne i}
 \frac{1-2y_i+y_iy_p}{y_i+y_p-2y_iy_p},
 \qquad
 \alpha_j=\frac{1+y_j}{y_j}
 \prod_{p\ne j}\frac{1-y_jy_p}{y_j-y_p}.
\end{equation*}
Put $s=(-1)^n$ and suppose $P=\eta\in\{1,-1\}$.

\begin{lemma}[Four-residue row identity]\label{lem:four-residue}
On $P=\eta=\pm1$,
\begin{equation}\label{eq:row-identity}
 \sum_j\overline K_{ij}\alpha_j
 =sb_i-s(1+\eta)L_i-\frac{1+s}{2}q^2H_i.
\end{equation}
Moreover,
\begin{equation}\label{eq:small-alpha}
 \sum_j\frac{\alpha_j}{1+y_j}=\frac{1+s}{2}-s\eta,
 \qquad
 \sum_j\frac{y_j\alpha_j}{1+y_j}=\frac{1-s}{2}.
\end{equation}
\end{lemma}

\begin{proof}
Fix $a=y_i$, let $Y_{\hat i}=Y\setminus\{y_i\}$, and consider
\begin{equation}\label{eq:charge-residue-fn}
 \mathcal R_i(\theta)
 =-
 \frac{N(a,\theta)}{(1+a)S(a,\theta)\theta(1-\theta^2)}
 \prod_{p\in Y_{\hat i}}\frac{1-\theta p}{\theta-p}.
\end{equation}
At the poles $\theta=y_j$, $j\ne i$, the residues are
$\overline K_{ij}\alpha_j$.  The remaining finite poles are
\[
 \theta=0,\qquad \theta=1,\qquad \theta=-1,\qquad
 \theta_0=\frac{a}{2a-1}.
\]
The explicit numerator \eqref{eq:N-explicit} gives
\[
 N(a,0)=1-a,\qquad
 N(a,1)=q^2a(1-a),
\]
\[
 N(a,-1)=-(3a-1)\bigl(2(1-a)+q^2a\bigr),
\]
and
\[
 N(a,\theta_0)
 =-\frac{(a-1)^2(a+1)(3a-1)}{(2a-1)^2}.
\]
Using $P=\eta$ one obtains
\[
 \operatorname*{Res}_{0}\mathcal R_i=s\eta L_i,
 \qquad
 \operatorname*{Res}_{1}\mathcal R_i=\frac{q^2}{2}H_i,
\]
\[
 \operatorname*{Res}_{-1}\mathcal R_i=sL_i+\frac{s q^2}{2}H_i,
 \qquad
 \operatorname*{Res}_{\theta_0}\mathcal R_i=-sb_i.
\]
Since $\mathcal R_i(\theta)=O(\theta^{-2})$, the residue at infinity vanishes.
Summing all residues proves \eqref{eq:row-identity}.

For \eqref{eq:small-alpha}, put
\[
 Q(\theta)=\prod_{p\in Y}\frac{1-\theta p}{\theta-p}.
\]
The residues at $\theta=y_j$ of
$Q(\theta)/(\theta(1-\theta^2))$ and $Q(\theta)/(1-\theta^2)$ are respectively
$\alpha_j/(1+y_j)$ and $y_j\alpha_j/(1+y_j)$.  Summing the remaining residues at
$0,1,-1$ gives the two identities in \eqref{eq:small-alpha}.
\end{proof}

\begin{lemma}[Pole cancellation]\label{lem:charge-pole}
The numerator of \eqref{eq:charge-recurrence} is divisible by $1-P^2$.
\end{lemma}

\begin{proof}
After the row scaling above, the recurrence coefficient satisfies
\begin{equation*}
 c_k\frac{\overline\Pi_{n-1}(Y_{\hat k})}{\overline\Pi_n(Y)}
 =(-1)^{n-k}P b_k.
\end{equation*}
Thus, on $P=\eta$, the pole numerator is a bordered Pfaffian with $b$ as one of
its borders.  Apply the simultaneous row/column operation encoded by the
coefficients $\alpha_j$.

Suppose first that $n$ is even.  Then \eqref{eq:row-identity} gives
\[
 b-\overline K\alpha=(1+\eta)L+q^2H.
\]
The other recurrence border is
\[
 D_\eta=(1+\eta)L+q(z+\eta z^{-1})H.
\]
If $\eta=1$, then $D_1=2L+q^2H$, so the two borders are equal.  If $\eta=-1$,
then $D_{-1}=q(z-z^{-1})H$, while the transformed $b$-border is $q^2H$; hence the
two borders are proportional.  Moreover, the second identity in
\eqref{eq:small-alpha} gives $\sum_i\alpha_iH_i=0$, so the auxiliary entry created
by the row operation vanishes.  In both cases the Pfaffian is zero.

Now suppose that $n$ is odd.  By Pfaffian multilinearity the two even-rank
inductive terms combine into one Pfaffian with variable borders
$L$ and $(1+\eta)qH$, mutual auxiliary entry $z+\eta z^{-1}$, and the recurrence
border $b$.  Since $s=-1$, \eqref{eq:row-identity} gives
\[
 b+\overline K\alpha=(1+\eta)L.
\]
The two identities in \eqref{eq:small-alpha} imply
\[
 \sum_iL_i\alpha_i=\eta-1,
 \qquad
 \sum_iH_i\alpha_i=1.
\]
For $\eta=1$ the transformed $b$-row, including its auxiliary entries, is twice
the existing $L$-row.  For $\eta=-1$ its variable part vanishes and the remaining
auxiliary part is proportional to the other auxiliary row.  Hence the Pfaffian
again vanishes.  Thus the recurrence numerator vanishes on both components
$P=1$ and $P=-1$, proving the required divisibility.
\end{proof}

\subsection{Interpolation uniqueness}
We include the degree bookkeeping.  First consider the Pfaffian candidate and
expand it along the row or column containing the index $n$.  If $n$ is paired
with a variable index $j$, the product of the corresponding prefactor and kernel
entries is
\[
 \frac{S(y_j,y_n)T(y_j,y_n)}{y_j-y_n}\,
 \overline K(y_j,y_n)
 =\frac{N(y_j,y_n)}{(1+y_j)(1+y_n)}.
\]
If $n$ is paired with an auxiliary border, the same denominator $1+y_n$ comes
from $L_n$, $H_n$, or $G_n$.  Put all terms over the additional common
denominator $\prod_{i<n}(y_n-y_i)$.  Before cancellation, the numerator has
degree at most $2n-1$ in $y_n$.  Alternation of the Pfaffian against the
Vandermonde denominator makes this numerator vanish at every $y_n=y_i$; after
division by the product of these $n-1$ linear factors, the degree is at most
$n$.  Thus the candidate has denominator at most $1+y_n$ and
$(1+y_n)Z_n$ has degree at most $n$.

The recurrence has the same bound.  Using the inductive degree estimate and
putting the terms of \eqref{eq:charge-recurrence} over a common denominator gives
\[
 (1+y_n)(1-P^2)Z_n
 =\frac{A(y_n)}{\prod_{i<n}(y_n-y_i)},
 \qquad \deg A\le 2n+1.
\]
The apparent factors $y_n-y_i$ cancel by symmetry, and
\cref{lem:charge-pole} removes the degree-two factor $1-P^2$ from the numerator.
After these two divisions the remaining degree is at most
$2n+1-(n-1)-2=n$.  Hence both the recurrence and the candidate have, as
functions of $y_n$, denominator at most $1+y_n$, and
\[
 (1+y_n)Z_n
\]
is a polynomial of degree at most $n$.  The $n+1$ values
\[
 y_n=0,\qquad y_n=1,\qquad y_n=y_i^{-1}\quad(1\le i<n)
\]
are given by \eqref{eq:charge-zero}--\eqref{eq:charge-recip} and agree for the two
sides.  Induction, starting with
\[
 Z_0=1,
 \qquad
 Z_1(z;y)=\frac{z+z^{-1}y}{1+y},
\]
proves \cref{thm:charge-pf}.

\section{The Pfaffian from \texorpdfstring{\cite{PR}}{[PR]} and the Chebyshev transform}\label{sec:pr-symmetric}
We now connect the charge identity of \cref{sec:charge} with the geometry of
symmetric matrix orbits.  From that section we use only two outputs: the
charge-to-Chebyshev symmetrization \eqref{eq:charge-cheb}--\eqref{eq:S-A} and
the charge-refined Pfaffian formula \cref{thm:charge-pf}.  Applying the former
to the latter will turn out to produce exactly the
\cite{PR} corank
generating series.

\subsection{Packaging the PR classes}
Define the corank generating series
\begin{equation*}
 \Phi_m(t;u)
 =\sum_{r=0}^m t^r\ssm(\Sigma^S_{m,r})(u),
 \qquad q=2t.
\end{equation*}
The symmetric $W$-function of Definition~5.3 in \cite{PR} can be repackaged
into the following single Pfaffian.

\begin{proposition}[PR generating Pfaffian]\label{prop:PR-pf}
If $m$ is odd,
\begin{equation*}
 \Phi_m(t;u)
 =R_m(u)\,
 \Pf\begin{pmatrix}
 0&(1+qu_j)_j\\
 -(1+qu_i)_i&C^{(q)}
 \end{pmatrix}.
\end{equation*}
If $m$ is even,
\begin{equation*}
 \Phi_m(t;u)
 =R_m(u)\,
 \Pf\begin{pmatrix}
 0&1&\one^T\\
 -1&0&(qu)^T\\
 -\one&-qu&C^{(q)}
 \end{pmatrix}.
\end{equation*}
Here $R_m$ and $C^{(q)}$ are given by
\eqref{eq:R-prefactor}--\eqref{eq:Cq}.
\end{proposition}

\begin{proof}
Write
\[
 C^{(q)}=C^{(0)}+q^2B,
 \qquad
 B_{ij}=u_i u_j\frac{u_i-u_j}{1+u_i+u_j},
\]
and put
\[
 \Gamma_m^S(u)=D_m^S(u)R_m(u)
 =\prod_{i<j}\frac{(u_i+u_j)(1+u_i+u_j)}{u_i-u_j}.
\]
Expand the Pfaffian multilinearly in the $q^2B$ entries and, when the number of
selected variable indices is odd, in the $q u_i$ boundary entries.  A term of
total degree $q^r$ is indexed by a subset $I\subset[m]$ with $|I|=r$: the
indices in $I$ are paired among themselves through $B$, with one of them paired
to the $q$-boundary when $r$ is odd, and the complementary indices are paired
through the $q=0$ Pfaffian.  The latter Pfaffian, multiplied by
$\Gamma^S_{m-r}(u_{\bar I})$, is the full-rank function
$W^S_{m-r}(u_{\bar I})$.

It remains to calculate the factor carried by $I$.  Set
\[
 \mathcal P_I=
 \begin{cases}
 \displaystyle
 \Pf\left[\frac{u_i-u_j}{1+u_i+u_j}\right]_{i,j\in I},
 &|I|\text{ even},\\[4mm]
 \displaystyle
 \Pf\begin{pmatrix}
 0&\one^T\\
 -\one&\left[\frac{u_i-u_j}{1+u_i+u_j}\right]_{i,j\in I}
 \end{pmatrix},
 &|I|\text{ odd}.
 \end{cases}
\]
Schur's Pfaffian identity, applied to the variables $1+2u_i$, gives
\[
 \mathcal P_I
 =\prod_{\substack{i<j\\i,j\in I}}
 \frac{u_i-u_j}{1+u_i+u_j}.
\]
For odd $|I|$, this bordered form is the limit of Schur's identity as an
auxiliary variable tends to infinity.  Since each selected vertex supplies one
factor $u_i$, the internal contribution to the coefficient of $t^r$ is
\begin{align*}
 &2^r
 \prod_{i\in I}u_i
 \prod_{\substack{i<j\\i,j\in I}}
 \frac{(u_i+u_j)(1+u_i+u_j)}{u_i-u_j}
 \mathcal P_I
 \\
 &\hspace{35mm}=
 \prod_{\substack{i\le j\\i,j\in I}}(u_i+u_j).
\end{align*}
Here the factor $2^r$ is the conversion from $q^r$ to $t^r$, because $q=2t$.
The quotient $\Gamma_m^S/(\Gamma_I^S\Gamma_{\bar I}^S)$ supplies the cross
factor.  We therefore obtain the explicit subset expansion
\begin{align*}
 D_m^S[t^r]\Phi_m
 =\sum_{\substack{I\subset[m]\\|I|=r}}
 &W^S_{m-r}(u_{\bar I})
 \prod_{\substack{i\le j\\i,j\in I}}(u_i+u_j)
 \prod_{\substack{i\in I\\j\in\bar I}}
 \frac{(u_i+u_j)(1+u_i+u_j)}{u_i-u_j}.
\end{align*}
With the inherited order on $I$ and $\bar I$, the Pfaffian sign is exactly the
sign of the displayed cross product.  The right-hand side is the symmetric
$W$-function $W^S_{m,r}$ of Definition~5.3 in \cite{PR}; Theorem~5.5 there
identifies it with $\csm(\Sigma^S_{m,r})$.  Dividing by
$D_m^S=c(S^2\CC^m)$ proves the proposition.
\end{proof}

\begin{remark}[Classical Pfaffians for symmetric and skew-symmetric degeneracy loci]
\label{rem:classical-pfaffians}
Pfaffians already occur in the classical formulas for the fundamental classes of
symmetric and skew-symmetric rank loci.  In our notation their staircase classes
are
\[
 [\overline{\Sigma^\wedge_{m,r}}]=s_{r-1,r-2,\ldots,1},
 \qquad
 [\overline{\Sigma^S_{m,r}}]=2^r s_{r,r-1,\ldots,1};
\]
see \cite{JLP,HT,FNR} and the broader degeneracy-locus framework in
\cite{Kazarian,AF}.  Equivalently, these are staircase Schur $P$- and $Q$-classes
(with the standard normalization), and hence admit the classical Schur-Pfaffian
formulas.  Kazarian developed Pfaffian formulas for Lagrangian and symmetric
degeneracy loci, and Anderson--Fulton place these Grassmannian formulas in a
larger type $B/C/D$ Pfaffian theory.

The Pfaffians in \cref{sec:skew,sec:charge,sec:pr-symmetric} are of a different
kind: they are indexed by the Chern roots $u_i$, rather than by the parts of the
staircase.  Nevertheless their lowest homogeneous terms recover the classical
classes above.  There is also a direct trace of the classical $P/Q$ distinction
in the root kernels.  If
\[
 B^\wedge_{ij}=\frac{u_i-u_j}{1+u_i+u_j}
\]
is the corank-selecting part of the skew Pfaffian in
\cref{sec:skew}, then the coefficient of $q^2$ in the symmetric bulk kernel
\eqref{eq:Cq} is
\[
 B^S_{ij}
 =u_i u_j\frac{u_i-u_j}{1+u_i+u_j}
 =u_i u_j B^\wedge_{ij}.
\]
Together with $q=2t$, the additional factors $u_i$ and $2$ are exactly what is
visible in passing from the skew staircase class $s_{\delta_{r-1}}$ to the
symmetric staircase class $2^r s_{\delta_r}$ at lowest degree.
\end{remark}

\subsection{The symmetrization square}
The charge Pfaffian \cref{thm:charge-pf} has exactly the same bulk kernel
$C^{(q)}$ as \cref{prop:PR-pf}.  Only the boundary differs.  The operator
$\mathscr S_z$ from \eqref{eq:S-op} transforms one into the other.
Indeed, for odd $m$,
\[
 \frac{(z+qu)+z(z^{-1}+qu)}{1+z}=1+qu.
\]
For even $m$, Pfaffian multilinearity in the second auxiliary row gives
\[
 \frac{(1+z^2)u+z(1+z^{-2})u}{1+z}
 =(z+z^{-1})u=qu.
\]
Hence
\begin{equation*}
 \mathscr S_z\mathcal A_m(z;u)=\Phi_m(t;u),
 \qquad 2t=z+z^{-1}.
\end{equation*}

The proof can be summarized by the following commutative square.
\begin{equation}\label{eq:four-vertex}
\begin{tikzcd}[column sep=large,row sep=large]
\displaystyle
\sum_{\ell(\lambda)\le m}
z^{\Delta_m(\lambda)}\st_\lambda
\arrow[r,leftrightarrow,"\text{Pfaffian interpolation}"]
\arrow[d,"\mathscr S_z"']
&
\text{charge Pfaffian}
\arrow[d,"\mathscr S_z"]
\\
\displaystyle
\sum_{\ell(\lambda)\le m}
V_{\kappa(\lambda)}(t)\st_\lambda
\arrow[r,equal]
&
\displaystyle
\sum_{r=0}^m t^r\frac{W^S_{m,r}}{c(S^2\CC^m)}
=\sum_{r=0}^m t^r\ssm(\Sigma^S_{m,r}).
\end{tikzcd}
\end{equation}

\begin{proof}[Proof of \cref{thm:sym-finite}]
The left vertical arrow in \eqref{eq:four-vertex} is
\eqref{eq:charge-cheb}; the top horizontal arrow is
\cref{thm:charge-pf}; the right vertical arrow is the boundary computation above;
and the lower-right expression is \cref{prop:PR-pf}.  This proves
\eqref{eq:sym-finite}.
\end{proof}

\begin{remark}
The appearance of Chebyshev polynomials is therefore not an a posteriori pattern
recognition.  The modified Robbins recurrence naturally sees the two charge states
$z$ and $z^{-1}$; the geometric PR class is obtained by the involutive boundary
symmetrization $\mathscr S_z$, and this symmetrization converts the charge of the
$2$-core into the Chebyshev polynomial $V_k$.
\end{remark}

\section{Probability consequences and comparison of the two geometries}\label{sec:prob-consequences}

We return to the random partition $\Lambda$ of \cref{thm:intro-prob}.  The two
matrix representations lead to two rather different random observables.

\subsection{Skew: an event}
For $r\equiv\epsilon\pmod2$, \cref{thm:intro-skew} gives
\[
 \ssm(\Sigma^\wedge_{\infty,r})(x)
 =\sum_{d_\epsilon(\lambda)=r}
 \PP(\Lambda=\lambda)
 =\PP(d_\epsilon(\Lambda)=r).
\]
Thus the stable skew orbit decomposition is literally a partition of the sample
space by the Maya-dimer random variable $d_\epsilon$.

\subsection{Symmetric: a Chebyshev observable}
Let
\[
 K=\kappa(\Lambda),
 \qquad
 \core(\Lambda)=\delta_K.
\]
Then \cref{thm:intro-sym} gives
\begin{equation}\label{eq:prob-cheb-final}
 \sum_{r\ge0}t^r\ssm(\Sigma^S_{\infty,r})(x)
 =\EE[V_K(t)].
\end{equation}
Equivalently, if
\[
 \pi_k=\PP(\core(\Lambda)=\delta_k)
 =\sum_{\core(\lambda)=\delta_k}\st_\lambda(x),
\]
then
\[
 \sum_{r\ge0}t^r\ssm(\Sigma^S_{\infty,r})(x)
 =\sum_{k\ge0}\pi_kV_k(t).
\]
Hence the stable symmetric SSM classes are a universal Chebyshev transform of the
$2$-core distribution.
This comparison can be summarized as
\[
 \begin{array}{ccl}
 \Lambda^2\CC^n
 &:&
 \ssm_r=\PP\{\text{Maya-dimer defect}=r\},\\[1mm]
 S^2\CC^n
 &:&
 \displaystyle\sum_rt^r\ssm_r
 =\EE\bigl[V_{\text{$2$-core length}}(t)\bigr].
 \end{array}
\]

\section{Further directions}
\subsection{Probability theory}
The probability interpretation raises several natural problems.  It would be interesting to describe the distributions of the Maya-dimer defect and the $2$-core directly from the stochastic six-vertex process, without first summing over endpoint partitions.  In the symmetric case the two-runner charge $\Delta_m$ has a particularly simple endpoint interpretation: in beta-number coordinates it is the difference between the numbers of exits in the two parity classes.  It is therefore natural to ask for contour, determinantal, or asymptotic
formulas for its distribution and for the Chebyshev moments
\eqref{eq:prob-cheb-final}.

A concrete instance of the same question is the distribution of the largest part.
The lattice underlying \cref{thm:intro-prob} is exactly a rational six-vertex
model with partial domain-wall boundary conditions, for which determinant
formulas are known both for the partition function itself and for boundary
one-point functions \cite{FW,MP}.  In our normalization the rightmost top exit is
$i_m=\lambda_1+m$, so such boundary one-point functions should translate directly
into an explicit formula for
\[
 \PP(\lambda_1=k)
 =\sum_{\lambda:\,\lambda_1=k}\st_\lambda(x_1,\ldots,x_m),
\]
presumably simplifying to a terminating hypergeometric-type expression in the
homogeneous case $x_1=\cdots=x_m$.

\subsection{Positivity in geometry}
The six-vertex realization
makes the coefficientwise monomial positivity of the finite matrix-Schubert CSM
numerator, $P_I\in\ZZ_{\ge0}[x_1,\ldots,x_m]$ of \cref{cor:monomial-positive},
transparent.  This is similar to, but distinct from, the positivity results
for CSM classes of ordinary Schubert cells, the Aluffi-Mihalcea conjecture \cite{AM} proved in \cite{Huh}.  It would be interesting to interpret our (monomial) positivity statements by presenting effective (torus equivariant) cycle representatives, for instance via a log-resolution or stratification compatible with the matrix-Schubert geometry.

\subsection{Singularity theory}
The geometrically relevant generalizations of degenerations of linear maps (the objects of this paper) are in two directions: (i) {\em higher degree} maps, and (ii) {\em diagrams} of linear maps (quivers).

Initial results in the nonlinear setting (the theory of SSM Thom polynomials) exhibit notable sparsity and degreewise sign coherence in
their $\st$-expansions, e.g.,
\begin{multline*}
\text{ssmTp}_{\CC[t]/(t^3)}^{l=0}=
(\st_{11}+2\st_2)
-
(6\st_{21}+6\st_{3})
+
(5\st_{211}+5\st_{22}+26\st_{31}+14\st_4)
\\
-(12\st_{221}+38\st_{311}+34\st_{32}+82\st_{41}+30\st_5)
+
\ldots.
\end{multline*}
For similar formulas and some conjectures, see \cite[Fig.~4 and Conj.~7.2]{RimanyiThom} and the \cite{TPP}. The structure suggests the possibility of probability-theoretic interpretations of SSM Thom polynomials.

\bibliographystyle{alpha}
\bibliography{refs}

\end{document}